\documentclass[a4paper]{article}
\usepackage{a4wide}
\usepackage[UKenglish]{babel}
\usepackage{amsfonts,amsmath,amsthm,mathtools,bbm}
\usepackage{graphicx,subfigure,epstopdf,xcolor}
	\graphicspath{{./}{figure/}{figure/2Devo/}}
	\DeclareGraphicsExtensions{.pdf}
\usepackage{paralist}
\usepackage[colorlinks=true,linkcolor=blue,citecolor=red]{hyperref}
\allowdisplaybreaks

\numberwithin{equation}{section}

\newcommand{\abs}[1]{\vert#1\vert}
\newcommand{\ave}[1]{\left\langle#1\right\rangle}

\newcommand{\cC}{\mathcal{C}}
\newcommand{\cD}{\mathcal{D}}

\newcommand{\e}{\varepsilon}
\newcommand{\cF}{\mathcal{F}}
\newcommand{\cL}{\mathcal{L}}
\newcommand{\norm}[1]{\Vert#1\Vert_\infty}
\newcommand{\rangeth}{I_\Theta}
\newcommand{\RR}{\mathbb{R}}
\newcommand{\vv}{\mathbf{v}}
\newcommand{\Var}{\operatorname{Var}}
\newcommand{\VV}{\mathbb{V}}
\newcommand{\Vx}{\mathbb{V}_x}
\newcommand{\Vy}{\mathbb{V}_y}
\newcommand{\ww}{\mathbf{w}}
\newcommand{\xx}{\mathbf{x}}
\newcommand{\yy}{\mathbf{y}}

\def\be{\begin{equation}}
\def\ee{\end{equation}}

\newtheorem{proposition}{Proposition}
\theoremstyle{remark}\newtheorem{remark}{Remark}

\begin{document}
\title{Hybrid stochastic kinetic description of two-dimensional traffic dynamics}
\author{	Michael Herty\thanks{Department of Mathematics, IGPM, RWTH Aachen University, Templergraben 55, 52062 Aachen, Germany.
			\texttt{herty@igpm.rwth-aachen.de }} \and
		Andrea Tosin\thanks{Department of Mathematical Sciences ``G. L. Lagrange'', Politecnico di Torino, Corso Duca degli Abruzzi 24, 10129 Turin, Italy.
			\texttt{andrea.tosin@polito.it}} \and
		Giuseppe Visconti\thanks{Department of Mathematics, IGPM, RWTH Aachen University, Templergraben 55, 52062 Aachen, Germany.
			\texttt{visconti@igpm.rwth-aachen.de}} \and
		Mattia Zanella\thanks{Department of Mathematical Sciences ``G. L. Lagrange'', Politecnico di Torino, Corso Duca degli Abruzzi 24, 10129 Turin, Italy.
			\texttt{mattia.zanella@polito.it}}
		}
\date{}

\maketitle

\begin{abstract}
In this work we present a two-dimensional kinetic traffic model which takes into account speed changes both when vehicles interact along the road lanes and when they change lane. Assuming that lane changes are less frequent than interactions along the same lane and considering that their mathematical description can be done up to some uncertainty in the model parameters, we derive a hybrid stochastic Fokker-Planck-Boltzmann equation in the quasi-invariant interaction limit. By means of suitable numerical methods, precisely structure preserving and direct Monte Carlo schemes, we use this equation to compute theoretical speed-density diagrams of traffic both along and across the lanes, including estimates of the data dispersion, and validate them against real data.

\medskip

\noindent{\bf Keywords:} Boltzmann and Fokker-Planck equations, uncertainty quantification, structure preserving schemes, fundamental diagrams, data dispersion \\

\noindent{\bf MSC:} 35Q20, 35Q70, 35Q84, 90B20
\end{abstract}


\section{Introduction}
\label{sec:Introduction}

In recent years the legacy of classical kinetic theory has found fruitful applications in the mathematical description of social phenomena~\cite{albi2016BOOKCH,carrillo2010MSSET,cordier2005JSP,furioli2017M3AS,pareschi2017PHYSA,toscani2006CMS}, including those, such as traffic flow of both vehicles and pedestrians, which mix mechanical and behavioural aspects of the agents~\cite{agnelli2015M3AS,degond2013KRM,delitala2007M3AS,fermo2013SIAP,festa2017KRM,herty2010KRM,klar1997JSP,puppo2016CMS}. For the sake of completeness, however, we mention that the mathematical modelling of vehicular traffic by means of methods of the kinetic theory has by now a quite long history dating back to the pioneering works~\cite{paveri1975TR,prigogine1961PROC,prigogine1971BOOK}.

The construction of mathematical models of these phenomena has to face the lack of fundamental principles and background theories: physical forces normally driving the dynamics in classical particle systems like gases and fluids are replaced by empirical interactions among the agents which often are known only statistically, cf. e.g.~\cite{ballerini2008AB}. Therefore models are in principle characterised by random inputs, such as e.g. uncertain parameters, which may greatly impact on the realism of the theoretical results with respect to the empirical observations. This is particularly true for models, such as the kinetic ones, which link the individual interactions among the agents to the collective patterns emerging from such interactions. Recent efforts in this direction exploit the Uncertainty Quantification (UQ) setting, see e.g.~\cite{albi2015MPE,dimarco2017CHAPTER,jin2018SEMASIMAI,tosin2017CMS_preprint} for an introduction. As a matter of fact, UQ methods for stochastic kinetic equations represent a fundamental step towards the actual validation of kinetic models against real data. Some approaches towards the incorporation of data in those models have been also undertaken recently, see e.g. \cite{fan2014NHM,herty2017PREPRINT}.

In this paper we propose a new kinetic traffic model, which takes into account speed changes due both to interactions among the vehicles along the road lanes and to lane changes. Although a few traffic models for lane changes are already available in the literature, cf. e.g.~\cite{illner2003CMS,moridpour2010TL}, here the novelty consists in the fact that our kinetic model allows us to study the fundamental diagrams of traffic both for the classical case of the flow of vehicles along the lanes and for the flow of vehicles across the lanes, which is instead less classical also from the empirical point of view.

In more detail, besides the acceleration and braking dynamics typical of one-dimensional traffic models along a lane, we suggest that microscopic vehicle dynamics across the lanes are simply a relaxation towards a desired lateral speed, which however is not known deterministically and hence, in our context, plays the role of the aforesaid stochastic parameter. After implementing such microscopic dynamical rules in a stochastic Boltzmann-type equation, owing to the empirical evidence that lane changes are much less frequent than one-to-one vehicle interactions along the lanes we exploit the quasi-invariant limit technique~\cite{toscani2006CMS} to finally derive a hybrid Fokker-Planck-Boltzmann equation for the the probability density of the vehicles. In this equation a nonlinear Fokker-Planck operator describes the speed variations along the lanes, whereas a Boltzmann-type collision operator takes into account the speed variations across the lanes. To the best of our knowledge this describes a novel approach to multilane traffic. It is in particular different from kinetic models where lane changing is considered as additional balance terms to a kinetic equation \cite{MR1761769}. In fact, the latter modelling does not allow one to account for the intrinsic dynamics across the lanes. 

In simplified cases, such as those of mean-field-type interactions among the vehicles, we obtain from the model analytical information on the large-time trend of the system. In particular, we are able to compute the asymptotic probability density of the cars and some of its relevant statistical moments, for instance the mean and the energy. In the general case, however, the large-time behaviour of the model is not known analytically. In order to investigate it accurately, and in particular to find the predicted fundamental diagrams of traffic, we build a suitable numerical scheme for the hybrid stochastic kinetic problem, which in particular extends second order Structure Preserving schemes for UQ available in the literature~\cite{dimarco2017CHAPTER,pareschi2017JSC} to fully nonlinear Fokker-Planck equations with non-vanishing diffusion. From  numerical simulations we observe that the average trend of our model reproduces correctly the fundamental diagrams of traffic both along and across the lanes. Moreover, the quantification of the uncertainty introduced by the stochastic parameter in the dynamics across the lanes proves to be essential in accounting at a theoretical level for the dispersion of the data around the mean normally observed in experimental fundamental diagrams.

Specifically, the rest of the paper is organised as follows. In Section~\ref{sec:MicroModel} we discuss the microscopic models of traffic dynamics along and across the road lanes, which are at the basis of our kinetic model. As usual in kinetic theory, we give them in the form of binary (i.e. one-to-one) interactions among the vehicles. In Section~\ref{sec:Boltzmann} we formulate the stochastic Boltzmann-type equation and we study, in a simplified setting, the evolution of some of its thermodynamic-like moments (mean speed and energy), which give insights into the macroscopic trends of the system. In Section~\ref{sec:hybridmodel} we derive the hybrid stochastic kinetic model and, again under suitable simplifying assumptions, we investigate its asymptotic distributions. In Section~\ref{sec:num} we build and test the numerical scheme for the hybrid problem, then we employ it to investigate, also by means of comparison with real data, the fundamental diagrams of traffic produced by the model in the general case. Finally, in Section~\ref{sec:conclusion} we summarise the main contributions of the work and briefly sketch research perspectives.

\section{Two-dimensional microscopic dynamics}
\label{sec:MicroModel}
Unlike most kinetic models of vehicular traffic available in the literature, which typically treat the flow of vehicles as one-dimensional, in this paper we consider the case of genuinely two-dimensional velocities describing the flow along the road and across the lanes. Consistently, the microscopic state of a vehicle will be the pair $\vv:=(v_x,\,v_y)$, where $v_x$ is the speed along the road ($x$-direction) and $v_y$ the lateral speed ($y$-direction). Notice that $v_x$ can be only positive, because the flow of vehicles in the longitudinal direction of a road is unidirectional, while $v_y$ can be either positive or negative, because lane changes are possible both leftwards and rightwards. Therefore we assume
$$ 0\leq v_x\leq 1, \qquad \abs{v_y}\leq \e, $$
where $ 0<\e\leq 1$ since lateral speeds are in general lower than longitudinal ones. We write
$$ \Vx:=[0,\,1], \qquad \Vy:=[-\e,\,\e] $$
for the domains of the two components of the velocity, which have to be understood as dimensionless and referred to suitable characteristic maximal values. The microscopic state space is therefore the set $\VV:=\Vx\times\Vy\subset\RR^2$.

The starting point of a kinetic model is the description of the microscopic speed transitions produced by binary interactions between any two vehicles. In our two-dimensional setting we need to design microscopic interactions both in the $x$-direction and in the $y$-direction to account for different dynamics in the two main directions of the flow. In particular, we assume that the interaction frequency \emph{across} the lanes (i.e. in the $y$-direction) is much smaller than  \emph{along}  lanes (i.e. in the $x$-direction) and, consistently, that the $x$-dynamics modify mainly the speed $v_x$ leaving $v_y$ unaltered, while the $y$-dynamics modify mainly the speed $v_y$ leaving $v_x$ unaltered.

\subsection{Microscopic rules for the $\xx$-dynamics}
\label{sec:MicroModelX}
Following~\cite{herty2010KRM,visconti2017MMS}, we assume that the post-interaction speed $v_x'$ in the $x$-direction is given by:
\begin{subequations}
\begin{equation}
	v_x'=\begin{cases}
			v_x+\alpha P(\rho)(V_A-v_x)+\sqrt{\alpha P(\rho)}D_A(v_x)\xi & \text{if } v_x<W_x \\[2mm]
			v_x+\alpha(1-P(\rho))(V_B-v_x)+\sqrt{\alpha(1-P(\rho))}D_B(v_x)\xi & \text{if } v_x>W_x, \\
		\end{cases}
	\label{eq:binary_x-v}
\end{equation}
where:
\begin{itemize}
\item $0<\alpha\leq 1$ is a constant weighting the strength of the interaction;
\item $P(\rho)\in [0,\,1]$ is the \emph{probability of accelerating} given as a function of the density $\rho$ of the vehicles, cf.~\cite{prigogine1971BOOK} and see below for a more detailed discussion;
\item $V_A$, $V_B$ are target speeds in acceleration and deceleration, respectively;
\item $\xi$ is a random variable modelling a stochastic fluctuation with zero mean and finite variance $\sigma^2>0$, and $D_A,\,D_B\geq 0$ are diffusion coefficients depending on the speed $v_x$ itself, see below.
\end{itemize}

From~\eqref{eq:binary_x-v} we see that the definition of $v_x'$ depends on the comparison between the current speed $v_x$ and a reference speed $W_x\in [0,\,1]$, that discriminates if the vehicle accelerates or brakes. Possible choices for $W_x$ are:
$$ W_x=w_x \quad\text{or}\quad W_x=u_x, $$
where $w_x$ is the $x$-component of the velocity $\ww:=(w_x,\,w_y)$ of a leading vehicle whereas $u_x$ denotes the mean speed of the flow in $x$-direction, we refer to~\cite{herty2010KRM} for an extensive discussion. If $W_x=w_x$ then we are in the case of genuine \emph{binary} interactions and we assume parallelly that
\begin{equation}
	w_x'=w_x,
	\label{eq:binary_x-w}
\end{equation}
\end{subequations}
i.e. that the $x$-speed of the leading vehicle remains unchanged after the interaction. Conversely, if $W_x=u_x$ we are in the case of the so-called \emph{mean-field} interactions, that can be regarded as an approximation of the previous ones, cf.~\cite{visconti2017MMS}.

The probability of accelerating $P=P(\rho)$ is in general a non-increasing function of the density $\rho$ of the vehicles. In more detail, assuming that $0\leq \rho\leq 1$ -- where $\rho=1$ is the dimensionless value corresponding to the maximum density that can be accommodated in a fully congested road (bumper-to-bumper traffic), one expects that $P\to 1^-$ when $\rho\to 0^+$ and that $P\to 0^+$ when $\rho\to 1^-$. The expression of $P$ that we consider here is in particular
\begin{equation}
	P(\rho):=1-\rho^\delta, \qquad \delta\geq 0.
	\label{eq:P}
\end{equation}

The target speeds $V_A$, $V_B$ describe instead the driving style of the individuals. For consistency, we require that $v_x<V_A\leq 1$ and that $0\leq V_B<v_x$. In~\cite{herty2010KRM,visconti2017MMS} several choices of $V_A$ and $V_B$ are discussed along with their influence on the structure of the resulting fundamental diagrams of traffic. In this paper we stick to the modelling of $V_A$ and $V_B$ introduced in~\cite{visconti2017MMS}, namely
\begin{equation}
	V_A:=\min\{v_x+\Delta{v},\,1\}, \qquad V_B:=P(\rho)W_x,
	\label{eq:VA.VB}
\end{equation}
where $\Delta{v}>0$ is a fixed parameter denoting the speed jump in acceleration while $W_x$ is the reference speed discussed above.

The local relevance of the stochastic fluctuation $\xi$, modelling random effects in the choice of the post-interaction speed by the drivers, is weighted by the diffusion coefficients $D_A,\,D_B$, that here we consider of the form
\begin{equation}
	\begin{array}{l}
		D_A(v_x):=\nu(v_x)(V_A-v_x)^\kappa \\[2mm]
		D_B(v_x):=\nu(v_x)(v_x-V_B)^\kappa,
	\end{array} 
	\qquad \text{with}\quad\nu(v_x):=v_x(1-v_x)\quad\text{and}\quad\kappa\geq 1,
	\label{eq:DA_DB}
\end{equation}
cf.~\cite{herty2010KRM}. In particular, the function $\nu$ makes the stocastic fluctuation vanish at the boundary of $\Vx$ (the $x$-speed domain), i.e. for $v_x=0$ and $v_x=1$.

For a general unbounded stochastic fluctuation $\xi\in\RR$ it may happen that the post-interaction speed $v_x'$ resulting from~\eqref{eq:binary_x-v} lies outside $\Vx$, implying that not all binary interactions are admissible. In order to prevent this it is sufficient to consider compactly supported stochastic fluctuations as stated in the following result.
\begin{proposition} \label{prop:xi.bounded}
If 
$$ \frac{\alpha (1-P(\rho))-1}{\sqrt{\alpha (1-P(\rho))}}\leq\abs{\xi}\leq\frac{1-\alpha P(\rho)}{\sqrt{\alpha P(\rho)}} $$
then $v_x'\in\Vx$ for all $v_x\in\Vx$.
\end{proposition}
\begin{proof}
Let us consider the case $v_x<W_x$ in~\eqref{eq:binary_x-v}. Since $0\leq V_A-v_x\leq 1$, $0\leq \nu(v_x)\leq 1$ and $\kappa\geq 1$ we have $D_A(v_x)\xi\leq(V_A-v_x)\abs{\xi}$, whence
$$ v_x'\leq\left(1-\alpha P(\rho)-\sqrt{\alpha P(\rho)}\abs{\xi}\right)v_x
	+\left(\alpha P(\rho)+\sqrt{\alpha P(\rho)}\abs{\xi}\right)V_A. $$
If $\alpha P(\rho)+\sqrt{\alpha P(\rho)}\abs{\xi}\leq 1$, i.e.
\begin{equation}
	\abs{\xi}\leq\frac{1-\alpha P(\rho)}{\sqrt{\alpha P(\rho)}},
	\label{eq:condition_xi}
\end{equation}
the right-hand side is a convex combination of $v_x$, $V_A$. This implies $v_x'\leq\max\{v_x,\,V_A\}=V_A\leq 1$.

On the other hand, since $\nu(v_x)\leq v_x$ we also have $D_A(v_x)\xi\geq -v_x\abs{\xi}$ and therefore
$$ v_x'\geq\left(1-\alpha P(\rho)-\sqrt{\alpha P(\rho)}\abs{\xi}\right)v_x+\alpha P(\rho)V_A, $$
which under~\eqref{eq:condition_xi} produces $v_x'\geq\alpha P(\rho)V_A\geq 0$.

Summarising, condition~\eqref{eq:condition_xi} guarantees that $v_x'\in\Vx$ for all $v_x<W_x$.

For the case $v_x>W_x$ in~\eqref{eq:binary_x-v} we proceed similarly using the fact that $\nu(v_x)\leq 1-v_x$, thereby deducing also the lower bound on $\abs{\xi}$.
\end{proof}

\subsection{Microscopic rules for the $\yy$-dynamics}
\label{sec:MicroModelY}
We model lane changing as a continuous process in an additional spatial dimension. This dimension is orthogonal to the driving direction and denoted by $y$. When vehicles move across the lanes (in $y$-direction) we consider the following dynamics:
\begin{equation}
	v_y'=v_y+\beta(u_x)(v_d(\theta)-v_y).
	\label{eq:binary_y}
\end{equation}
Notice that~\eqref{eq:binary_y} accounts neither for binary nor for mean-field interactions. Since lane changes are much less frequent than interactions along the main stream of traffic, the rule~\eqref{eq:binary_y} simply assumes that the lateral speed of the vehicles relaxes towards a \emph{desired speed} $v_d\in\Vy$, which will be presumably close to zero. However, in order to add realism to the very basic dynamics~\eqref{eq:binary_y}, we refrain from fixing deterministically the value of $v_d$ and assume instead that it depends on a random parameter $\theta\in\rangeth\subseteq\RR$.
 We will come back more precisely to this aspect in the next sections.

The term $\beta(u_x)$ in~\eqref{eq:binary_y} models the relaxation rate towards $v_d$. Specifically, it depends on $u_x$, which is the mean speed in the $x$-direction, so that the post-interaction speed $v_y'$ across the lanes is affected by the traffic flow along the lanes. Thinking of $v_d$ close on average to zero, a conceivable choice is
$$ \beta(u_x)\propto u_x, $$
meaning that the faster the flow along the lanes the faster the relaxation towards $v_d$, namely towards zero, across the lanes, consistently with the intuition that lane changes are not necessary if the traffic is sufficiently fluent in the driving directions.

Similarly to the $x$-dynamics discussed in Section~\ref{sec:MicroModelX}, also for~\eqref{eq:binary_y} we need to ensure that $v_y'\in\Vy$ for all $v_y\in\Vy$. The following result holds.
\begin{proposition} \label{th:prop2}
If $\beta:[0,\,1]\to [0,\,1]$ then $v_y'\in\Vy$ for all $v_y\in\Vy$.
\end{proposition}
\begin{proof}
By rewriting~\eqref{eq:binary_y} as $v_y'=(1-\beta(u_x))v_y+\beta(u_x)v_d(\theta)$ we see that, under the assumption $0\leq\beta(u_x)\leq 1$, the post-interaction speed $v_y'$ is a convex combination of $v_y$, $v_d(\theta)\in\Vy$. Hence the thesis easily follows from the convexity of $\Vy$.
\end{proof}

\section{Stochastic Boltzmann-type description}
\label{sec:Boltzmann}
The interaction rules~\eqref{eq:binary_x-v}-\eqref{eq:binary_x-w},~\eqref{eq:binary_y} can be encoded in a kinetic Boltzmann-type description of the dynamics. This is particularly useful to study the asymptotic macroscopic trends of the system, possibly taking advantage of suitable scaling and limit procedures.

For a certain realisation $\theta\in\rangeth$ of the random parameter appearing in~\eqref{eq:binary_y}, let $f=f(\vv,\,t;\,\theta):\VV\times [0,\,+\infty)\to\RR_+$ be the kinetic distribution function such that $f(\vv,\,t;\,\theta)\,d\vv$ is the fraction of vehicles which at time $t\geq 0$ have a microscopic speed in an infinitesimal volume of the state space $\VV$ centred at $\vv$. Since $\theta$ is a constant parameter in each $y$-interaction, whose precise value is however unknown, we proceed along the lines of the so-called \emph{Uncertainty Quantification} (UQ): we first consider the family of all possible dynamics of the system for $\theta\in\rangeth$, which amounts to regarding $f$ as parametrised by $\theta$; next we average their outputs according to the probability distribution of $\theta$, say $h=h(\theta):\rangeth\to\RR_+$ such that $\int_{\rangeth}h(\theta)\,d\theta=1$. We refer to~\cite{tosin2017CMS_preprint} for more details.

Under the interaction schemes set forth in Section~\ref{sec:MicroModel}, the time evolution of $f$ is given by the following Boltzmann-type kinetic equation in weak form:
\begin{align}
	\begin{aligned}[b]
		\frac{d}{dt}\int_\VV\varphi(\vv)f(\vv,\,t;\,\theta)\,d\vv &=
			\frac{\rho}{2}\ave{\int_\VV\int_\VV\left(\varphi(\vv'_x)-\varphi(\vv)\right)f(\vv,\,t;\,\theta)f(\ww,\,t;\,\theta)\,d\vv\,d\ww} \\
		&\phantom{=} +\gamma\rho\int_\VV\left(\varphi(\vv_y')-\varphi(\vv)\right)f(\vv,\,t;\,\theta)\,d\vv,
	\end{aligned}
	\label{eq:boltzmann}
\end{align}
where
\begin{itemize}
\item $\varphi:\VV\to\RR$ is a test function, i.e. any observable function of the microscopic state $\vv$;
\item the first term at the right-hand side accounts for the interactions in the $x$-direction which leave the speed $v_y$ unaltered; in particular, $\vv'_x:=(v'_x,\,v_y)$ with $v'_x$ given by~\eqref{eq:binary_x-v}. The coefficient $\rho/2$ is the interaction rate, which is supposed to be proportional to the density of the vehicles and, in particular, takes into account the asymmetric form of the interactions~\eqref{eq:binary_x-v}-\eqref{eq:binary_x-w}, cf.~\cite{tosin2017IFAC_preprint};
\item $\ave{\cdot}$ denotes the expectation with respect to the stochastic fluctuation $\xi$, cf.~\eqref{eq:binary_x-v};
\item the second term at the right-hand side accounts for speed changes in the $y$-direction which leave the speed $v_x$ unaltered; in particular, $\vv'_y:=(v_x,\,v'_y)$ with $v'_y$ given by~\eqref{eq:binary_y}. The coefficient $\gamma\rho$ is the interaction rate, with $0<\gamma\ll 1$ modelling the much lower frequency of the interactions across the lanes with respect to those along the lanes.
\end{itemize}

It is worth pointing out that~\eqref{eq:boltzmann} is a \emph{stochastic} Boltzmann-type equation, because it is parametrised by the random parameter $\theta$. From the knowledge of the kinetic distribution function $f$ one can compute $\theta$-expected quantities, such as the expected distribution function and its $\theta$-variance:
\begin{equation}
	\bar{f}(\vv,\,t):=\int_{\rangeth}f(\vv,\,t;\,\theta)h(\theta)\,d\theta, \qquad
		\Var_\theta(f)(\vv,\,t):=\int_{\rangeth}f^2(\vv,\,t;\,\theta)h(\theta)\,d\theta-{\bar{f}}^2(\vv,\,t).
	\label{eq:fbar.var}
\end{equation}
Similarly, from the thermodynamic-like moments of $f$ parametrised by $\theta$:
$$ M_\varphi(t;\,\theta):=\int_\VV\varphi(\vv)f(\vv,\,t;\,\theta)\,d\vv $$
one can recover the average  expected $\vv$-moments and their $\theta$-variance:
\begin{align*}
	& \bar{M}_\varphi(t):=\int_{\rangeth}M_\varphi(t;\,\theta)h(\theta)\,d\theta=\int_\VV\varphi(\vv)\bar{f}(\vv,\,t)\,d\vv, \\
	& \Var_\theta(M_\varphi)(t):=\int_{\rangeth}M^2(t;\,\theta)h(\theta)\,d\theta-{\bar{M}}^2_\varphi(t),
\end{align*}
which are useful tools for quantifying the uncertainty induced in the system dynamics by the random parameter $\theta$. Notice that from~\eqref{eq:boltzmann} it is in general not possible to derive a closed equation for $\bar{f}(\vv,\,t)$ by simply integrating both sides with respect to $h(\theta)\,d\theta$.

\subsection{Evolution of the macroscopic quantities}
\label{eq:macroevolution}
First of all, from~\eqref{eq:boltzmann} with $\varphi(\vv)=1$ we obtain that the integral of $f$ with respect to $\vv$ is conserved in time for all $\theta\in\rangeth$. Hence, if $f(\cdot,\,t;\,\theta)$ is chosen to be a probability density at $t=0$ it will be so for all $t>0$. The physical counterpart of this fact is the conservation of the mass of vehicles, whose density is fixed by the parameter $\rho\in [0,\,1]$ appearing in~\eqref{eq:binary_x-v} and~\eqref{eq:boltzmann}.

Let us now consider any $p$-th order moment, $p\in\mathbb{N}$, of $f$ in the $x$-direction, which amounts to taking $\varphi(\vv)=v_x^p$. Plugging into~\eqref{eq:boltzmann} we get
\begin{equation}
	\frac{d}{dt}\int_{\VV}v_x^pf(\vv,\,t;\,\theta)\,d\vv=
		\frac{\rho}{2}\ave{\int_\VV\int_\VV\left((v_x')^p-v_x^p\right)f(\vv,\,t;\,\theta)f(\ww,\,t;\,\theta)\,d\vv\,d\ww},
	\label{eq:moments-x}
\end{equation}
because $\varphi(\vv_y')-\varphi(\vv)=v_x^p-v_x^p=0$. Similarly, if we consider any $p$-th order moment of $f$ in the $y$-direction, i.e. if we take $\varphi(\vv)=v_y^p$, we discover
\begin{equation}
	\frac{d}{dt}\int_{\VV}v_y^pf(\vv,\,t;\,\theta)\,d\vv=
		\gamma\rho\int_\VV\left((v_y')^p-v_y^p\right)f(\vv,\,t;\,\theta)\,d\vv,
	\label{eq:moments-y}
\end{equation}
because now $\varphi(\vv_x')-\varphi(\vv)=v_y^p-v_y^p=0$.

This argument implies that the evolution of the macroscopic quantities in the single directions of the traffic flow may be obtained from~\eqref{eq:boltzmann} by considering separately the two collision operators at the right-hand side. Notice, however, that it is in general not possible to reconstruct the kinetic distribution function $f(\cdot,\,t;\,\theta)$ on the whole space $\VV$ of the microscopic states by taking in~\eqref{eq:boltzmann} test functions which depend on only one of the two speeds, namely by looking at the dynamics in only one direction.

\subsubsection{Macroscopic $\xx$-dynamics}
\label{sec:macro_x}
We now investigate in more detail the evolution equations of some macroscopic quantities in the $x$-direction. Precisely, we consider~\eqref{eq:binary_x-v}-\eqref{eq:binary_x-w} and~\eqref{eq:moments-x} in the simplified setting
$$ V_A=u_x, \quad V_B=u_x, \quad W_x=u_x, $$
which makes possible some explicit analytical computations.

The time evolution of the $x$-mean speed $u_x=u_x(t;\,\theta)$ results from~\eqref{eq:moments-x} with $p=1$. In particular, recalling that the stochastic fluctuation $\xi$ is a centred random variable, we obtain
\begin{align*}
	\frac{du_x}{dt} &= \frac{\alpha\rho}{2}\left(P(\rho)\int_{-\e}^\e\int_0^{u_x}(u_x-v_x)f(\vv,\,t;\,\theta)\,dv_x\,dv_y\right. \\
	&\phantom{=} \left.+(1-P(\rho))\int_{-\e}^\e\int_{u_x}^1(u_x-v_x)f(\vv,\,t;\,\theta)\,dv_x\,dv_y\right).
\end{align*}
Now, observing that
\begin{align*}
	& \int_{-\e}^\e\int_0^{u_x}(u_x-v_x)f(\vv,\,t;\,\theta)\,dv_x\,dv_y+\int_{-\e}^\e\int_{u_x}^1(u_x-v_x)f(\vv,\,t;\,\theta)\,dv_x\,dv_y \\
	&= \int_{\VV}(u_x-v_x)f(\vv,\,t;\,\theta)\,d\vv=0,
\end{align*}
we get
\begin{equation*}
	\frac{du_x}{dt}=
	\begin{cases}
		\dfrac{\alpha\rho}{2}(2P(\rho)-1)\displaystyle\int_{-\e}^\e\int_0^{u_x}(u_x-v_x)f(\vv,\,t;\,\theta)\,dv_x\,dv_y \\[3mm]
		\dfrac{\alpha\rho}{2}(1-2P(\rho))\displaystyle\int_{-\e}^\e\int_{u_x}^1(u_x-v_x)f(\vv,\,t;\,\theta)\,dv_x\,dv_y,
	\end{cases}
\end{equation*}
whence finally, summing the two equations,
$$ \frac{du_x}{dt}=\frac{\alpha\rho}{4}(2P(\rho)-1)\int_{\VV}\abs{u_x-v_x}f(\vv,\,t;\,\theta)\,d\vv. $$

By defining the marginal distribution $f_x(v_x,\,t;\,\theta):=\int_{-\e}^\e f(\vv,\,t;\,\theta)\,dv_y$, we notice that at the right-hand side it results $\int_{\VV}\abs{u_x-v_x}f(\vv,\,t;\,\theta)\,d\vv=0$ if and only if $f_x(v_x,\,t;\,\theta)=\delta_{u_x}(v_x)$. Therefore, if $\rho(2P(\rho)-1)\ne 0$, the only steady state in the $x$-direction, which allows for a stationary mean speed, is the \emph{synchronised traffic} with all the vehicles travelling at the same speed~\cite{kerner2002MCM}. Conversely, for $f_x(v_x,\,t;\,\theta)\ne\delta_{u_x}(v_x)$ (and $\rho\ne 0$) the mean speed either increases or decreases in time depending on the sign of $2P(\rho)-1$. In particular, it increases for $P(\rho)>\frac{1}{2}$, which defines the so-called \emph{free phase} of traffic when the vehicle density is small (recall that the mapping $\rho\mapsto P(\rho)$ is non-increasing); whereas it decreases for $P(\rho)<\frac{1}{2}$, which defines the so-called \emph{congested phase} of traffic when the vehicle density is large. The value $\rho=\rho_c$ such that $P(\rho_c)=\frac{1}{2}$ is called the \emph{critical density}. For the function~\eqref{eq:P} it results, for instance, $\rho_c=(1/2)^{1/\delta}$, which is consistent with the values found in~\cite{puppo2016CMS,puppo2017CMS,puppo2017KRM} for different kinetic models of traffic flow.

In order to further explore the macroscopic trends of the model it is useful to investigate also the evolution of the energy along the lanes, say $E_x=E_x(t;\,\theta)$, namely the second order $x$-moment of $f$ obtained by taking $p=2$ in~\eqref{eq:moments-x}. Since the complete equation for $E_x$ is quite complicated, we conveniently resort to a particular limit procedure, called the \emph{quasi-invariant interaction limit}~\cite{toscani2006CMS}, which allows us to grasp the essential time-asymptotic behaviour of $E_x$. Specifically, in~\eqref{eq:binary_x-v} we consider the regime of weak but frequent interactions. This corresponds to taking $\alpha$ small (notice that $\alpha$ tunes both the strength of the speed variation and the variance of the stochastic fluctuation) and to simultaneously scaling the time as $\tau:=\alpha t$. In practice, we pass from the characteristic $t$-scale of single microscopic interactions to a larger time scale defined by the variable $\tau$. Introducing the scaled kinetic distribution function $g(\vv,\,\tau;\,\theta):=f(\vv,\,\tau/\alpha;\,\theta)$ and noticing that $\partial_\tau g=\frac{1}{\alpha}\partial_t f$ we obtain from~\eqref{eq:moments-x} with $p=2$ the equation
\begin{align*}
	\frac{dE_x}{d\tau} &=
		\frac{\rho}{2\alpha}\ave{\int_\VV\int_\VV\left((v_x')^2-v_x^2\right)g(\vv,\,\tau;\,\theta)g(\ww,\,\tau;\,\theta)\,d\vv\,d\ww},
\intertext{whence, using~\eqref{eq:binary_x-v} together with $\ave{\xi}=0$, $\ave{\xi^2}=\sigma^2$ and letting $\alpha\to0^+$,}
	&= \frac{\sigma^2\rho}{2}\left(P(\rho)\int_{-\e}^\e\int_0^{u_x}D_A^2(v_x)g(\vv,\,\tau;\,\theta)\,dv_x\,dv_y\right. \\
	&\phantom{=} \left.+(1-P(\rho))\int_{-\e}^\e\int_{u_x}^1D_B^2(v_x)g(\vv,\,\tau;\,\theta)\,dv_x\,dv_y\right) \\
	&\phantom{=} +\frac{\rho}{2}(u_x^2-E_x)-\frac{\rho}{2}(1-2P(\rho))\int_\VV v_x\abs{u_x-v_x}g(\vv,\,\tau;\,\theta)\,d\vv.
\end{align*}
In particular, in the absence of stochastic fluctuation ($\sigma^2=0$) this equation specialises as
$$ \frac{dE_x}{d\tau}=\frac{\rho}{2}(u_x^2-E_x)-\frac{\rho}{2}(1-2P(\rho))\int_\VV v_x\abs{u_x-v_x}g(\vv,\,\tau;\,\theta)\,d\vv. $$

The term $u_x^2-E_x$ at the right-hand side is the opposite of the variance of the microscopic speeds in the $x$-direction, therefore it is non-positive. Moreover, for $P(\rho)\leq\frac{1}{2}$, namely in the congested phase of traffic, also the second term at the right-hand side is non-positive, which makes the energy on the whole non-increasing in time. Conversely, for $P(\rho)>\frac{1}{2}$, namely in the free phase of traffic, the second term at the right-hand side is non-negative, thus in principle the energy may not be monotonic in this case. This implies that the convergence to the steady state $E_x\to u_x^2$, consistent with the asymptotic state of synchronised traffic discussed before, is in general smoother in the congested than in the free phase of traffic.

Finally, we stress that also in the case $\sigma^2>0$ the full equation of $E_x$ gives an asymptotic trend of the energy consistent with the synchronised traffic (i.e. $E_x\to u_x^2$) thanks to the fact that with the definition~\eqref{eq:DA_DB} it results $D_A(u_x)=D_B(u_x)=0$.

\subsubsection{Macroscopic $\yy$-dynamics}
\label{sec:macro_y}
We now study the evolution of the mean speed $u_y=u_y(t;\,\theta)$ and energy $E_y=E_y(t;\,\theta)$ of traffic in the $y$-direction taking advantage of~\eqref{eq:moments-y} complemented with the microscopic dynamics~\eqref{eq:binary_y}.

For $p=1$ we get
$$ \frac{du_y}{dt}=\gamma\rho\beta(u_x)(v_d(\theta)-u_y), $$
therefore asymptotically ($\frac{du_y}{dt}\to 0$) it results $u_y\to v_d(\theta)$, consistently with microscopic relaxation dynamics towards the desired speed $v_d$.

To investigate the asymptotic trend of the energy $E_y$ it is convenient to resort also in this case to the quasi-invariant interaction limit. For this, we assume for instance $\beta(u_x)=\beta_0 u_x$, $0<\beta_0\leq 1$, and we consider the regime of small $\beta_0$. By scaling the time as $\tau:=\beta_0 t$ and the distribution function as $g(\vv,\,\tau;\,\theta):=f(\vv,\,\tau/\beta_0;\,\theta)$ we obtain from~\eqref{eq:moments-y}
$$ \frac{d}{d\tau}\int_{\VV}v_y^pg(\vv,\,\tau;\,\theta)\,d\vv=
	\frac{\gamma\rho}{\beta_0}\int_{\VV}\left((v_y')^p-v_y^p\right)g(\vv,\,\tau;\,\theta)\,d\vv, $$
whence, for $p=2$ and using~\eqref{eq:binary_y} in the limit $\beta_0\to 0^+$,
$$ \frac{dE_y}{d\tau}=2\gamma\rho u_x(v_d(\theta)u_y-E_y), $$
which asymptotically ($\frac{dE_y}{d\tau}\to 0$) produces $E_y\to v_d(\theta)u_y\to v_d^2(\theta)$. This implies that the speed variance in the $y$-direction tends to zero, namely that $f_y(v_y,\,t;\,\theta)\to\delta_{v_d(\theta)}(v_y)$, where $f_y$ denotes the marginal distribution $f_y(v_y,\,t;\,\theta):=\int_0^1f(\vv,\,t;\,\theta)\,dv_x$.

\section{Hybrid kinetic model}
\label{sec:hybridmodel}
Considering again the full Boltzmann-type equation~\eqref{eq:boltzmann} with general terms $V_A$, $V_B$, $W_x$ in~\eqref{eq:binary_x-v}, we now use the quasi-invariant interaction limit introduced in Section~\ref{sec:macro_x} to derive a hybrid kinetic model under the assumption of different interaction frequency among the vehicles along and across the lanes, cf. Section~\ref{sec:MicroModel}. The advantage of the resulting model is that it is simpler than~\eqref{eq:boltzmann} but still preserving the original asymptotic dynamics and steady states at both the kinetic and the macroscopic levels. In more detail, the nonlinear integral collision operator in the $x$-direction (first term at the right-hand side of~\eqref{eq:boltzmann}) is replaced by a \emph{Fokker-Planck-type transport-diffusion differential operator} which describes the \emph{mean-field} effect of the frequent interactions among the vehicles along the lanes. Parallelly, the linear collision operator in the $y$-direction (second term at the right-hand side of~\eqref{eq:boltzmann}) remains to describe the rare interactions among the vehicles across the road lanes.

As already mentioned, the quasi-invariant interaction limit has been introduced in~\cite{toscani2006CMS}, see also~\cite{pareschi2013BOOK}, as an asymptotic procedure reminiscent of the \emph{grazing collision limit} in classical kinetic theory~\cite{desvillettes1992TTSP,diperna1988CMP,pareschi2003NM,villani1999M2NA}. Since then it has been widely used in the literature to study the large-time trends of e.g. traffic flow models~\cite{herty2010KRM,visconti2017MMS}, crowd dynamics models~\cite{festa2017KRM}, opinion formation models~\cite{albi2016BOOKCH}, socio-economic models~\cite{cordier2005JSP,furioli2017M3AS}.

\subsection{The Fokker-Planck-Boltzmann model}
\label{sec:FP-Boltz}
The regime that we want to study is characterised by a small value of the parameter $\alpha$ in~\eqref{eq:binary_x-v}, corresponding to weak interactions in the $x$-direction, and by a simultaneously small value of the parameter $\gamma$ in~\eqref{eq:boltzmann}, corresponding to rare interactions in the $y$-direction with respect to those in the $x$-direction. As before, we introduce the time scale $\tau:=\alpha t$, where the frequency of the $x$-binary interactions raises to $O(1/\alpha)$, and we scale the distribution function as $g(\vv,\,\tau;\,\theta):=f(\vv,\,\tau/\alpha;\,\theta)$. Notice that for $\alpha$ small we have $t=\tau/\alpha$ large, hence the limit $\alpha\to 0^+$ describes the asymptotic trend of $f$. On the other hand, in view of the previous definition, the asymptotic trend of $f$ is well approximated by that of $g$.

Since $\partial_\tau g=\frac{1}{\alpha}\partial_tf$, from~\eqref{eq:boltzmann} we get
\begin{align}
	\begin{aligned}[b]
		\frac{d}{d\tau}\int_\VV\varphi(\vv)g(\vv,\,\tau;\,\theta)\,d\vv &=
			\frac{\rho}{2\alpha}\ave{\int_\VV\int_\VV\left(\varphi(\vv'_x)-\varphi(\vv)\right)g(\vv,\,\tau;\,\theta)g(\ww,\,\tau;\,\theta)\,d\vv\,d\ww} \\
		&\phantom{=} +\frac{\gamma}{\alpha}\rho\int_\VV\left(\varphi(\vv_y')-\varphi(\vv)\right)g(\vv,\,\tau;\,\theta)\,d\vv.
	\end{aligned}
	\label{eq:boltzmann.scaled-1}
\end{align}
Let us pick a smooth test function with compact support $\varphi\in C^3_c(\VV)$. Expanding the difference $\varphi(\vv'_x)-\varphi(\vv)$ at the right-hand side we have
\begin{align*}
	\varphi(\vv'_x)-\varphi(\vv) &= \partial_{v_x}\varphi(\vv)(v_x'-v_x)+\frac{1}{2}\partial^2_{v_x}\varphi(\vv)(v_x'-v_x)^2
		+\frac{1}{6}\partial^3_{v_x}\varphi(\bar{\vv}_x)(v_x'-v_x)^3,
\end{align*}
where $\bar{\vv}_x:=(\bar{v}_x,\,v_y)$ is a point such that $\min\{v_x,\,v_x'\}<\bar{v}_x<\max\{v_x,\,v_x'\}$. Using the expression of $v_x'$ given in~\eqref{eq:binary_x-v}, with $\ave{\xi}=0$, $\ave{\xi^2}=\sigma^2$, and plugging this expansion into~\eqref{eq:boltzmann.scaled-1} we discover
\begin{align}
	\begin{aligned}[b]
		\frac{d}{d\tau}\int_\VV\varphi(\vv)g(\vv,\,\tau;\,\theta)\,d\vv &=
			-\frac{\rho}{2}\int_\VV\int_\VV\partial_{v_x}\varphi(\vv)L(v_x,\,V_A,\,V_B,\,W_x)g(\vv,\,\tau;\,\theta)g(\ww,\,\tau;\,\theta)\,d\vv\,d\ww \\
		&\phantom{=} +\frac{\sigma^2\rho}{4}\int_\VV\int_\VV\partial_{v_x}^2\varphi(\vv)D^2(v_x,\,W_x)
				g(\vv,\,\tau;\,\theta)g(\ww,\,\tau;\,\theta)\,d\vv\,d\ww \\
		&\phantom{=} +R_\alpha(\varphi) \\			
		&\phantom{=} +\frac{\gamma}{\alpha}\rho\int_\VV\left(\varphi(\vv_y')-\varphi(\vv)\right)g(\vv,\,\tau;\,\theta)\,d\vv,
	\end{aligned}
	\label{eq:boltzmann.scaled-2}
\end{align}
where we have denoted for brevity
\begin{align}
	\begin{aligned}[c]
		L(v_x,\,V_A,\,V_B,\,W_x) &:=
			\begin{cases}
				P(\rho)(v_x-V_A) & \text{if } v_x<W_x \\
				(1-P(\rho))(v_x-V_B) & \text{if } v_x>W_x,
			\end{cases} \\[2mm]
		D(v_x,\,W_x) &:=
			\begin{cases}
				\sqrt{P(\rho)}D_A(v_x) & \text{if } v_x<W_x \\
				\sqrt{1-P(\rho)}D_B(v_x) & \text{if } v_x>W_x.
			\end{cases}
	\end{aligned}
	\label{eq:LD}
\end{align}
Furthermore the term $R_\alpha(\varphi)$ is
\begin{align*}
	R_\alpha(\varphi) &:= -\frac{\alpha\rho}{4}\int_\VV\int_\VV\partial_{v_x}^2\varphi(\vv)L(v_x,\,V_A,\,V_B,\,W_x)
		g(\vv,\,\tau;\,\theta)g(\ww,\,\tau;\,\theta)\,d\vv\,d\ww \\
	&\phantom{:=} -\frac{\alpha^2\rho}{12}\int_\VV\int_\VV\partial_{v_x}^3\varphi(\bar{\vv}_x)L^3(v_x,\,V_A,\,V_B,\,W_x)
		g(\vv,\,\tau;\,\theta)g(\ww,\,\tau;\,\theta)\,d\vv\,d\ww \\
	&\phantom{:=} -\frac{\alpha\sigma^2\rho}{4}\int_\VV\int_\VV\partial_{v_x}^3\varphi(\bar{\vv}_x)L(v_x,\,V_A,\,V_B,\,W_x)D^2(v_x,\,W_x)
		g(\vv,\,\tau;\,\theta)g(\ww,\,\tau;\,\theta)\,d\vv\,d\ww \\
	&\phantom{:=} +\frac{\sqrt{\alpha}\rho}{12}\ave{\xi^3}\int_\VV\int_\VV\partial_{v_x}^3\varphi(\bar{\vv}_x)D^3(v_x,\,W_x)
		g(\vv,\,\tau;\,\theta)g(\ww,\,\tau;\,\theta)\,d\vv\,d\ww
\end{align*}
and is such that
\begin{align*}
	\abs{R_\alpha(\varphi)} &\leq \frac{\alpha\rho}{4}\norm{\partial_{v_x}^2\varphi}\norm{L}
		+\frac{\alpha^2\rho}{12}\norm{\partial_{v_x}^3\varphi}\norm{L}^3 \\
	&\phantom{\leq} +\frac{\alpha\sigma^2\rho}{4}\norm{\partial_{v_x}^3\varphi}\norm{L}\norm{D}^2
		+\frac{\sqrt{\alpha}\rho}{12}\ave{\abs{\xi}^3}\norm{\partial_{v_x}^3\varphi}\norm{D}^3,
\end{align*}
where $\norm{\cdot}$ is the $\infty$-norm in $\VV$. Since $L$ and $D$ are bounded and $\xi$ has finite moments of any order thanks to Proposition~\ref{prop:xi.bounded}, we deduce $R_\alpha(\varphi)\to 0$ for $\alpha\to 0^+$.

Finally, taking the limit $\alpha\to 0^+$, $\gamma\to 0^+$ in~\eqref{eq:boltzmann.scaled-2} and assuming $\gamma/\alpha=O(1)$, i.e. $\gamma/\alpha\to \mu>0$, we obtain
\begin{align}
	\begin{aligned}[b]
		\frac{d}{d\tau}\int_\VV\varphi(\vv)g(\vv,\,\tau;\,\theta)\,d\vv &=
			-\int_\VV\partial_{v_x}\varphi(\vv)\cL[g](v_x,\,\tau;\,\theta)g(\vv,\,\tau;\,\theta)d\vv \\
		&\phantom{=} +\frac{\sigma^2}{2}\int_\VV\partial_{v_x}^2\varphi(\vv)\cD[g](v_x,\,\tau;\,\theta)g(\vv,\,\tau;\,\theta)\,d\vv \\
		&\phantom{=} +\mu\rho\int_\VV\left(\varphi(\vv_y')-\varphi(\vv)\right)g(\vv,\,\tau;\,\theta)\,d\vv,
	\end{aligned}
	\label{eq:boltzmann.limit}
\end{align}
where we have denoted
\begin{align}
	\begin{aligned}[c]
		\cL[g](v_x,\,\tau;\,\theta) &:= \frac{\rho}{2}\int_{\VV}L(v_x,\,V_A,\,V_B,\,W_x)g(\ww,\,\tau;\,\theta)\,d\ww, \\[2mm]
		\cD[g](v_x,\,\tau;\,\theta) &:= \frac{\rho}{2}\int_{\VV}D^2(v_x,\,W_x)g(\ww,\,\tau;\,\theta)\,d\ww.
	\end{aligned}
	\label{eq:LDcal}
\end{align}
Integrating back by parts the right-hand side of~\eqref{eq:boltzmann.limit} and recalling the compactness of the support of $\varphi$ we see that, owing to the arbitrariness of $\varphi$,~\eqref{eq:boltzmann.limit} is a weak form of the equation
\begin{equation}
	\partial_\tau g=\partial_{v_x}\left(\cL[g]g+\frac{\sigma^2}{2}\partial_{v_x}(\cD[g]g)\right)+\mu\rho Q_y(g),
	\label{eq:FP-Boltz}
\end{equation}
where
$$ Q_y(g)=Q_y(g)(\vv,\,\tau;\,\theta)=\frac{1}{1-\beta(u_x)}g('\vv_y,\,\tau;\,\theta)-g(\vv,\,\tau;\,\theta) $$
is the strong form of the Boltzmann-type collision operator in the $y$-direction. Specifically, $'\vv_y:=(v_x,\,'v_y)$ denotes the pre-interaction velocity yielding $\vv=(v_x,\,v_y)$ as post-interaction velocity according to~\eqref{eq:binary_y}, i.e. $'v_y=\frac{v_y-\beta(u_x)v_d(\theta)}{1-\beta(u_x)}$, while the coefficient $\frac{1}{1-\beta(u_x)}$ is the Jacobian of the transformation~\eqref{eq:binary_y}.

Equation~\eqref{eq:FP-Boltz} represents our \emph{hybrid kinetic model}, featuring at the right-hand side the Fokker-Planck-type operator
$$ \partial_{v_x}\left(\cL[g]g+\frac{\sigma^2}{2}\partial_{v_x}(\cD[g]g)\right) $$
for the frequent vehicle interactions in the $x$-direction complemented with the linear collision operator $Q_y(g)$ for the less frequent speed changes in the $y$-direction. The constant $\mu>0$ permits to tune the relative importance of the two terms.

\begin{remark}
Equation~\eqref{eq:boltzmann.limit}, hence~\eqref{eq:FP-Boltz}, has been obtained from~\eqref{eq:boltzmann.scaled-2} assuming that $\gamma/\alpha=O(1)$ for $\alpha,\,\gamma\to 0^+$. Other asymptotic regimes may be considered as well, among which we mention in particular the one with $\gamma/\alpha=o(1)$. It corresponds to interactions across the lanes so rare that for large times one recovers a classical one-dimensional traffic model with only $x$-dynamics along the lanes (in practice,~\eqref{eq:FP-Boltz} without the collision term $Q_y(g)$). Clearly, the choice leading to~\eqref{eq:FP-Boltz} is the one which guarantees a correct balance between the two contributions, thereby allowing one to study genuinely two-dimensional traffic dynamics with the proper frequencies.
\end{remark}

\subsection{Asymptotic distributions}
\label{sec:asympt.distr}
As stated at the beginning of Section~\ref{sec:hybridmodel}, model~\eqref{eq:FP-Boltz} preserves the macroscopic trends of the original Boltzmann-type model~\eqref{eq:boltzmann}. In particular, the large-time behaviour of $g$ is the same as that of $f$ under the quasi-invariant interaction limit, cf. Section~\ref{sec:FP-Boltz}. Owing to the results of Section~\ref{sec:macro_y}, this allows us to conclude immediately that $g_y(v_y,\,\tau;\,\theta)\to\delta_{v_d(\theta)}(v_y)$ for $\tau\to+\infty$, where $g_y(v_y,\,\tau;\,\theta):=\int_0^1g(\vv,\,\tau;\,\theta)\,dv_x$.

More in general, assuming for simplicity that $\beta(u_x)\equiv\beta_0>0$ in~\eqref{eq:binary_y}, so that the microscopic $x$- and $y$-dynamics are decoupled, we can look for asymptotic distributions $g^\infty=g^\infty(\vv;\,\theta)$ of the form
$$ g^\infty(\vv;\,\theta)=g^\infty_x(v_x)g^\infty_y(v_y;\,\theta). $$
Plugging this representation into~\eqref{eq:FP-Boltz} yields
\begin{equation}
	g^\infty_y\partial_{v_x}\left(\cL[g^\infty_x]g^\infty_x+\frac{\sigma^2}{2}\partial_{v_x}(\cD[g^\infty_x]g^\infty_x)\right)
		+\mu\rho g^\infty_xQ_y(g^\infty_y)=0,
	\label{eq:FP-Boltz.stationary}
\end{equation}
where
\begin{align*}
	\cL[g^\infty_x](v_x) &= \frac{\rho}{2}\int_0^1L(v_x,\,V_A,\,V_B,\,W_x)g^\infty_x(w_x)\,dw_x \\[2mm]
	\cD[g^\infty_x](v_x) &= \frac{\rho}{2}\int_0^1D^2(v_x,\,W_x)g^\infty_x(w_x)\,dw_x \\[2mm]
	Q_y(g^\infty_y)(v_y;\,\theta) &= \frac{1}{1-\beta_0}g^\infty_y('v_y;\,\theta)-g^\infty_y(v_y;\,\theta).
\end{align*}

Besides the already determined $g^\infty_y(v_y;\,\theta)=\delta_{v_d(\theta)}(v_y)$, which is such that $Q_y(g^\infty_y)=0$ (in the proper weak sense, cf.~\eqref{eq:boltzmann.limit}), from~\eqref{eq:FP-Boltz.stationary} we see that the asymptotic marginal distribution $g^\infty_x$ in the $x$-direction is determined by setting
$$ \cL[g^\infty_x]g^\infty_x+\frac{\sigma^2}{2}\partial_{v_x}(\cD[g^\infty_x]g^\infty_x)=0. $$
The way in which this equation can be solved may be strictly dependent on the choice of the terms $V_A$, $V_B$, $W_x$ in~\eqref{eq:binary_x-v}. For instance, from Section~\ref{sec:macro_x} we know that for $V_A=V_B=W_x=u_x$ we should expect $g^\infty_x(v_x;\,\theta)=\delta_{u_x}(v_x)$ (again in the proper weak sense, cf.~\eqref{eq:boltzmann.limit}). In a more general case, when we only assume $W_x=u_x$ and admit that $V_A$, $V_B$ may be functions of $v_x$, i.e. $V_A=V_A(v_x)$ and $V_B=V_B(v_x)$, following~\cite{visconti2017MMS} we determine
\begin{equation}
	g^\infty_x(v_x)=
		\begin{cases}
			C_A\left(\dfrac{V_A(u_x)-u_x}{V_A(v_x)-v_x}\right)^2
				\exp{\left(-\dfrac{2}{\sigma^2}\displaystyle{\int_{v_x}^{u_x}\frac{1}{V_A(v)-v}\,dv}\right)} & \text{if } v_x<u_x \\[5mm]
			C_B\left(\dfrac{u_x-V_B(u_x)}{v_x-V_B(v_x)}\right)^2
				\exp{\left(-\dfrac{2}{\sigma^2}\displaystyle{\int_{u_x}^{v_x}\frac{1}{v-V_B(v)}\,dv}\right)} & \text{if } v_x>u_x,
		\end{cases}
	\label{eq:g^infty_x}
\end{equation}
where $C_A,\,C_B>0$ are normalisation constants to be fixed in such a way that $\int_0^1g^\infty_x(v_x)\,dv_x=1$ and $\int_0^1v_xg^\infty_x(v_x)\,dv_x=u_x$.

\begin{remark}
We refer to the already mentioned paper~\cite{visconti2017MMS} for detailed expressions of $g^\infty_x$ in case of several choices of $V_A$, $V_B$ including~\eqref{eq:VA.VB}. We simply remark that for $V_A=V_B=u_x$ equation~\eqref{eq:g^infty_x} gives $g^\infty_x(v_x)=0$ for $v_x\neq u_x$, which is indeed consistent with the true distributional solution $g^\infty_x(v_x)=\delta_{u_x}(v_x)$.
\end{remark}

\section{Numerical results}
\label{sec:num}
In this section we present a numerical scheme for solving the hybrid stochastic Fokker-Planck-Boltzmann traffic equation~\eqref{eq:FP-Boltz} along the lines of Uncertainty Quantification (UQ). Among the popular numerical methods in the literature for UQ we recall here in particular stochastic collocation methods, stochastic Galerkin schemes and multi-level Monte Carlo methods, see e.g.~\cite{dimarco2017CHAPTER,mishra2013CHAPTER,xiu2005SISC}. In the sequel we will specifically consider stochastic collocation methods. They are based on introducing a discretisation $\{\theta_k\}_{k=0}^M\subset\rangeth$ of the uncertain parameter $\theta$ and then in solving, by means of well-established deterministic algorithms, a family of $M+1$ equations of the form~\eqref{eq:FP-Boltz}, each for a fixed value $\theta=\theta_k$. Their solutions $\{g(\vv,\,\tau;\,\theta_k)\}_{k=0}^M$ can finally be post-processed to obtain statistical information at both the kinetic and the macroscopic level with respect to $\theta$, cf. Section~\ref{sec:Boltzmann}. The collocation nodes $\theta_k$ are typically chosen according to Gaussian quadrature rules and consistently with the probability distribution of $\theta$, see~\cite{xiu2010BOOK,xiu2002SISC}.

\medskip

As a preliminary step to the numerical solution of~\eqref{eq:FP-Boltz}, we consider the following dimensional splitting: on a certain time interval $[\tau^n,\,\tau^{n+1}]$, with $\tau^n:=n\Delta{\tau}$, $n\in\mathbb{N}$ and $\Delta{\tau}>0$ fixed, we first solve the Fokker-Planck step
\begin{equation}
	\begin{cases}
		\partial_\tau\tilde{g}=\partial_{v_x}\left(\cL[\tilde{g}]\tilde{g}+\dfrac{\sigma^2}{2}\partial_{v_x}(\cD[\tilde{g}]\tilde{g})\right),
			& v_x\in\Vx,\ \tau^{n}<\tau\leq\tau^{n+1/2} \\
		\tilde{g}(\vv,\,\tau^n;\,\theta)=g(\vv,\,\tau^n;\,\theta)
	\end{cases}
	\label{eq:split.FP}
\end{equation}
for all $v_y\in\Vy$ (regarded as a parameter); then we solve the Boltzmann step as
\begin{equation}
	\begin{cases}
		\partial_\tau g=\mu\rho Q_y(g), & v_y\in\Vy,\ \tau^{n+1/2}<\tau\leq\tau^{n+1} \\[2mm]
		g(\vv,\,\tau^{n+1/2};\,\theta)=\tilde{g}(\vv,\,\tau^{n+1/2};\,\theta)
	\end{cases}
	\label{eq:split.Boltz}
\end{equation}
for all $v_x\in\Vx$ (regarded as a parameter). This process may be iterated to obtain the numerical solution of the initial equation at each time step.

To approximate numerically~\eqref{eq:split.FP} we adopt Structure Preserving methods which have been recently developed in~\cite{pareschi2017JSC}, see also~\cite{dimarco2017CHAPTER,tosin2017CMS_preprint}. Conversely, to solve~\eqref{eq:split.Boltz} we use direct Monte Carlo methods, see e.g.~\cite{pareschi2001ESAIMP,pareschi2013BOOK}. In particular, we employ stratified sampling methods to extract at each time step from $\tilde{g}$ the particle ensemble to be evolved in~\eqref{eq:split.Boltz}, see~\cite{pareschi2005ESAIMP}.

\subsection{Structure Preserving methods for nonlinear Fokker-Planck equations}
The derivation of Structure Preserving schemes in the fully nonlinear case (i.e., for a Fokker-Planck equation in which the diffusion coefficient depends on the unknown distribution function itself) follows from the approaches described in~\cite{buet2010CMS,chang1970JCP,larsen1985JCP} and has been further investigated, in both the deterministic and stochastic settings, in the recent works~\cite{dimarco2017CHAPTER,pareschi2017JSC}. 

We observe that for each $\theta_k$, $k=0,\,\dots,\,M$, the Fokker-Planck equation in~\eqref{eq:split.FP} may be written in flux form
$$ \partial_\tau\tilde{g}(\vv,\,\tau;\,\theta_k)=\partial_{v_x}\cF[\tilde{g}](\vv,\,\tau;\,\theta_k), $$
where the flux is
\begin{equation}
	\cF[\tilde{g}](\vv,\,\tau;\,\theta_k):=\cC[\tilde{g}](\vv,\,\tau;\,\theta_k)\tilde{g}(\vv,\,\tau;\,\theta_k)
		+\cD[\tilde{g}](\vv,\,\tau;\,\theta_k)\partial_{v_x}\tilde{g}(\vv,\,\tau;\,\theta_k)
	\label{eq:an_flux}
\end{equation}
with $\cC[\tilde{g}]:=\cL[\tilde{g}]+\partial_{v_x}\cD[\tilde{g}]$ and $\cL$, $\cD$ are defined in~\eqref{eq:LDcal}.

We introduce the uniform grids $\{v_{x,i}\}_{i=1}^{N_x}\subset\Vx$, $\{v_{y,j}\}_{j=1}^{N_y}\subset\Vy$ and let $\Delta{v_x}:=v_{x,i+1}-v_{x,i}>0$, $\Delta{v_y}:=v_{y,j+1}-v_{y,j}>0$ constant. We denote as usual $v_{x,i+1/2}:=v_{x,i}+\frac{1}{2}\Delta{v_x}$ and consider the conservative discretisation
\begin{equation}
	\dfrac{d}{d\tau}\tilde{g}_{i,j}^k(\tau)=\frac{\cF_{i+1/2,j}^k[\tilde{g}]-\cF_{i-1/2,j}^k[\tilde{g}]}{\Delta{v_x}},
		\qquad i=1,\,\dots,\,N_x,
	\label{eq:cons.discr}
\end{equation}
where, for each $\tau>0$, $\tilde{g}_{i,j}^k(\tau)\approx\tilde{g}(v_{x,i},\,v_{y,j},\,\tau;\,\theta_k)$ while $\cF_{i+1/2,j}^k[\tilde{g}]$ is the numerical flux that here we take of the form (cf.~\eqref{eq:an_flux})
\begin{equation}
	\cF_{i+1/2,j}^k[\tilde{g}]=\hat{\cC}^k_{i+1/2,j}\hat{g}_{i+1/2,j}^k+\frac{\cD^k_{i+1/2,j}}{2}\frac{\tilde{g}_{i+1,j}^k-\tilde{g}_{i,j}^k}{\Delta{v_x}}.
	\label{eq:numflux}
\end{equation}

In~\eqref{eq:numflux} we set in particular
$$ \hat{g}_{i+1/2,j}^k:=(1-\delta_{i+1/2,j}^k)\tilde{g}_{i+1,j}^k+\delta_{i+1/2,j}^k\tilde{g}_{i,j}^k, $$
which is a convex combination of the values of $\tilde{g}^k$ in the two adjacent cells $i$, $i+1$ provided $0\leq\delta_{i+1/2,j}^k\leq 1$. The standard approach based on central difference is obtained taking $\delta^k_{i+1/2,j}=\frac{1}{2}$ for all $i,\,j$ and $\hat{\cC}^k_{i+1/2,j}:=\cC[\tilde{g}](v_{x,i+1/2},\,v_{y,j},\,\tau;\,\theta_k)$. Setting in particular
\begin{equation}
	\hat{\cC}^k_{i+1/2,j}=\frac{\cL^k_{i+1/2,j}+(\partial_{v_x}\cD)^k_{i+1/2,j}}{\Delta{v_x}}
	\label{eq:Ctilde}
\end{equation}
we obtain explicitly 
\begin{equation}
	\delta_{i+1/2,j}^k=\frac{1}{\lambda_{i+1/2,j}^k}+\frac{1}{1-\exp(\lambda_{i+1/2,j}^k)}, \qquad
		\lambda_{i+1/2,j}^k:=\frac{\Delta{v_x}\hat{\cC}^k_{i+1/2,j}}{\cD^k_{i+1/2,j}}
	\label{eq:deltaCC}
\end{equation}
and the following result holds, cf.~\cite{pareschi2017JSC}: 
\begin{proposition}
The numerical flux~\eqref{eq:numflux} with $\hat{\cC}^k_{i+1/2,j}$, $\delta_{i+1/2,j}^k$ given by~\eqref{eq:Ctilde}-\eqref{eq:deltaCC} vanishes when the exact flux~\eqref{eq:an_flux} vanishes in $[v_{x,i},\,v_{x,i+1}]\subset\Vx$. Moreover, $\delta^k_{i+1/2,j}\in [0,\,1]$ for all $i,\,j$ for every fixed $k$.
\end{proposition}

\begin{remark}
In the limit case of vanishing diffusion ($\cD=0$) we obtain the weights
$$ \delta_{i+1/2,j}^k=
	\begin{cases}
		0 & \text{if } \cL_{i+1/2,j}^k>0 \\
		1 & \text{if } \cL_{i+1/2,j}^k<0
	\end{cases}
$$
and the scheme reduces to a first order upwind scheme for the corresponding continuity equation.
\end{remark}

The introduced scheme preserves the asymptotic solutions of the Fokker-Planck equation in~\eqref{eq:split.FP} with second order accuracy. In the case of linear diffusion, i.e. for $\cD$ independent of $\tilde{g}$, it captures such solutions with arbitrary accuracy, see~\cite{dimarco2017CHAPTER,pareschi2017JSC}. Furthermore, for general strong stability preserving and high-order semi-implicit methods it is possible to prove the non-negativity of the numerical solution without any restrictions on $\Delta{v_x}$ but with suitable restrictions on the time step $\Delta{\tau}$.

For the explicit-in-time scheme deduced from~\eqref{eq:cons.discr}, i.e.
$$ \tilde{g}_{i,j}^{k,n+1/2}=\tilde{g}_{i,j}^{k,n}+\frac{\Delta{\tau}}{\Delta{v_x}}\left(\cF_{i+1/2,j}^{k,n}[\tilde{g}]
	-\cF_{i-1/2,j}^{k,n}[\tilde{g}]\right), $$
where $\tilde{g}_{i,j}^{k,n}\approx\tilde{g}(v_{x,i},\,v_{y,j},\,n\Delta{\tau};\,\theta_k)$, the following result holds:
\begin{proposition} \label{prop:ex}
Let
$$ \Delta{\tau}\leq\frac{\Delta v_x^2}{2\left(\max\limits_{i,j,k}\abs{\hat{\cC}_{i+1/2,j}^k}\Delta{v_x}
	+\max\limits_{i,j,k}\cD_{i+1/2,j}^k\right)}. $$
Then the explicit-in-time scheme is positivity preserving, i.e. $\tilde{g}_{i,j}^{k,n+1/2}\geq 0$ if $\tilde{g}_{i,j}^{k,n}\geq 0$ for all $i,\,j$.
\end{proposition}

To avoid parabolic time step restrictions typical of the explicit schemes, such as the one in Proposition~\ref{prop:ex}, it is possible to resort to a semi-implicit scheme:
$$ \tilde{g}_{i,j}^{k,n+1/2}=\tilde{g}_{i,j}^{k,n}+\frac{\Delta{\tau}}{\Delta{v_x}}\left(\check{\cF}_{i+1/2,j}^{k,n+1/2}[\tilde{g}]
	-\check{\cF}_{i-1/2,j}^{k,n+1/2}[\tilde{g}]\right) $$
with 
$$ \check{\cF}_{i+1/2,j}^{k,n+1/2}:=\hat{\cC}^{k,n}_{i+1/2,j}
	\left[\left(1-\delta_{i+1/2,j}^{k,n}\right)\tilde{g}^{k,n+1/2}_{i+1,j}+\delta_{i+1/2,j}^{k,n}\tilde{g}_{i,j}^{k,n+1/2}\right]
		+\cD_{i+1/2,j}^n\frac{\tilde{g}_{i+1,j}^{k,n+1/2}-\tilde{g}_{i,j}^{k,n+1/2}}{\Delta v_x}. $$
In this case the following result holds:
\begin{proposition} \label{prop:SI_dt}
Let
$$ \Delta{\tau}\leq\frac{\Delta v_x}{2\max\limits_{i,j,k}\abs{\hat{\cC}_{i+1/2,j}^{k,n}}}. $$
Then the semi-implicit scheme is positivity preserving.
\end{proposition}

We omit the proofs of Propositions~\ref{prop:ex},~\ref{prop:SI_dt}, which are reminiscent of similar ones proposed in~\cite{pareschi2017JSC}. It is also possible to prove that the SP scheme dissipates the numerical entropy under specific assumptions, see~\cite{dimarco2017CHAPTER} for details.

\bigskip

\begin{figure}[!t]
\centering
\subfigure[$\sigma^2=10$]{\includegraphics[scale=0.45]{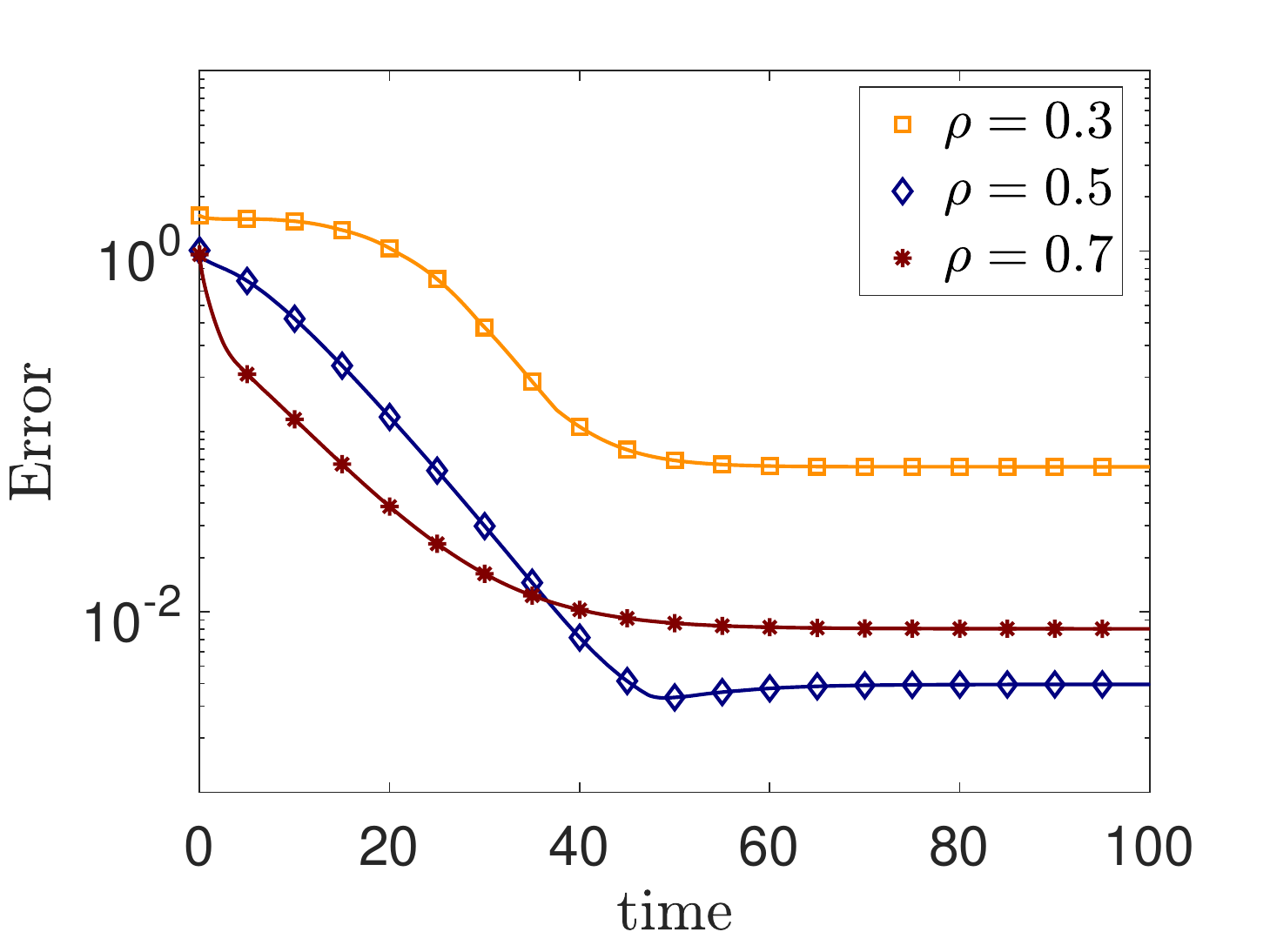}}
\subfigure[$\sigma^2=15$]{\includegraphics[scale=0.45]{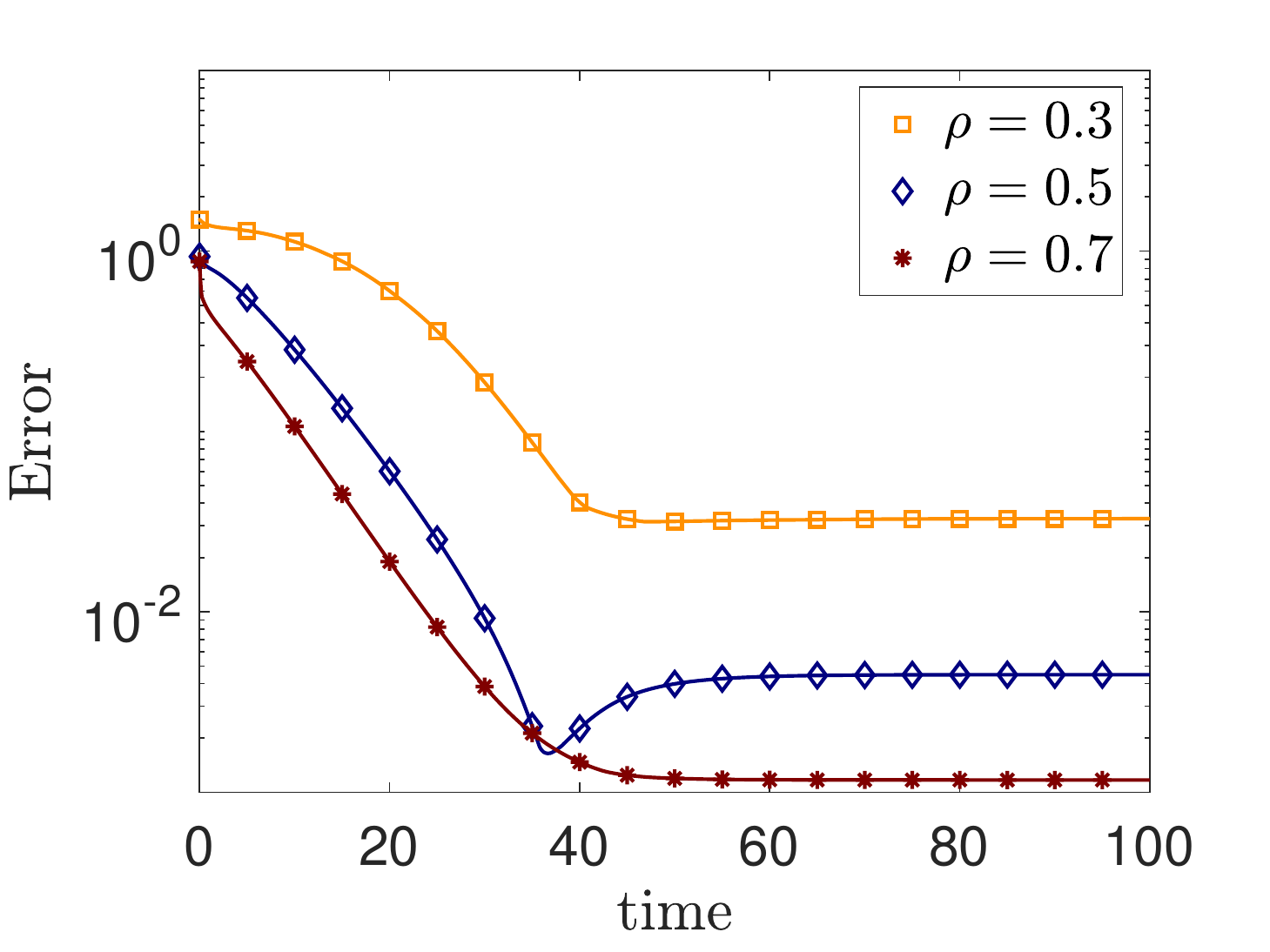}}
\caption{Evolution of the relative $L^1$ error on the numerical solution of the Fokker-Planck equation in~\eqref{eq:split.FP} computed with $N_x=41$ grid points with respect to a reference solution obtained with $N_x=321$ grid points. Time integration has been performed in the interval $[0,\,100]$ with the semi-implicit method with $\Delta{\tau}=\Delta{v_x}/\sigma^2$.}
\label{fig:rel_errorFP}
\end{figure}

\begin{figure}[!t]
\centering
\subfigure[$\rho=0.3,\,\sigma^2=10$]{
\includegraphics[scale=0.32]{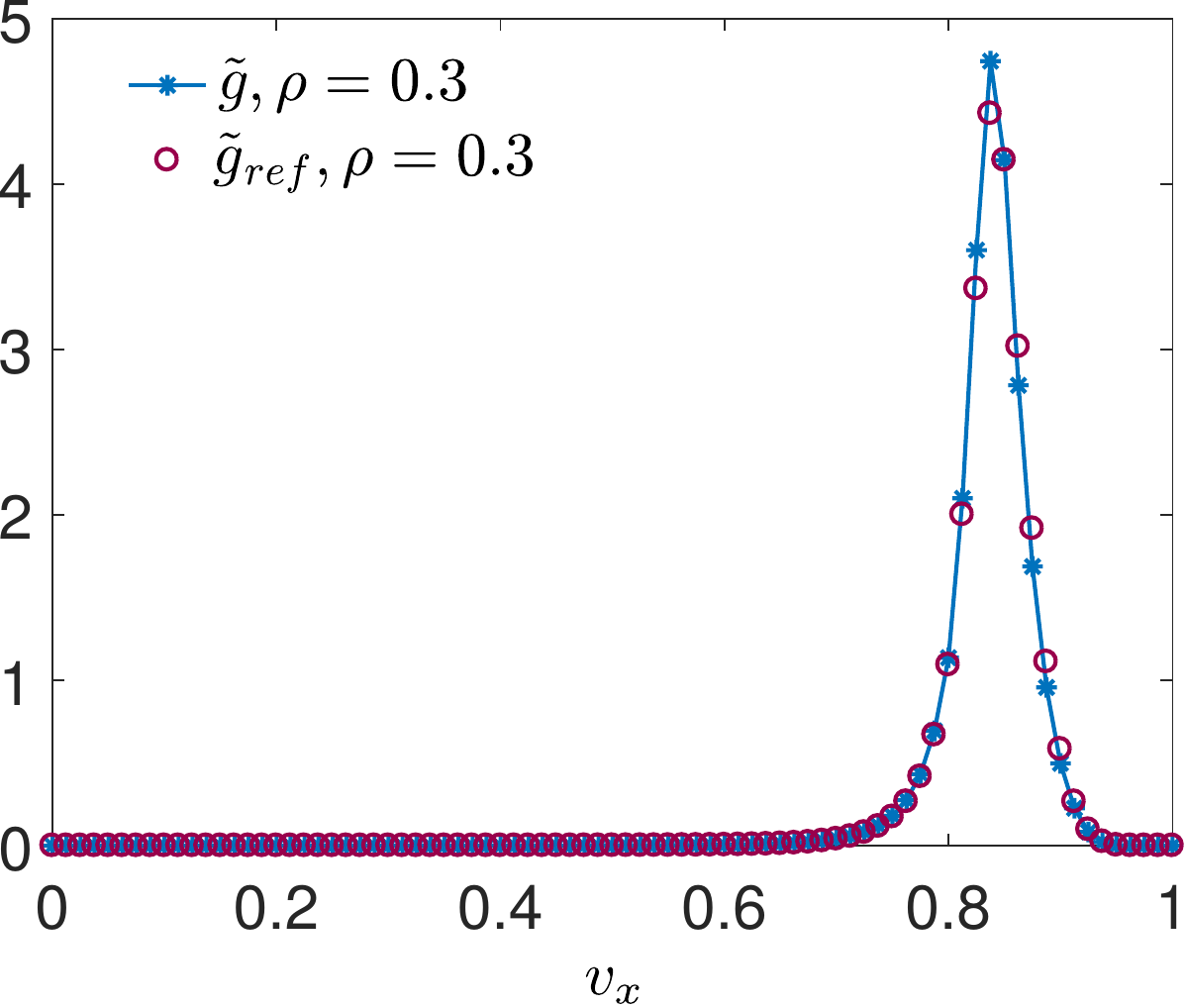}}
\subfigure[$\rho=0.5,\,\sigma^2=10$]{
\includegraphics[scale=0.32]{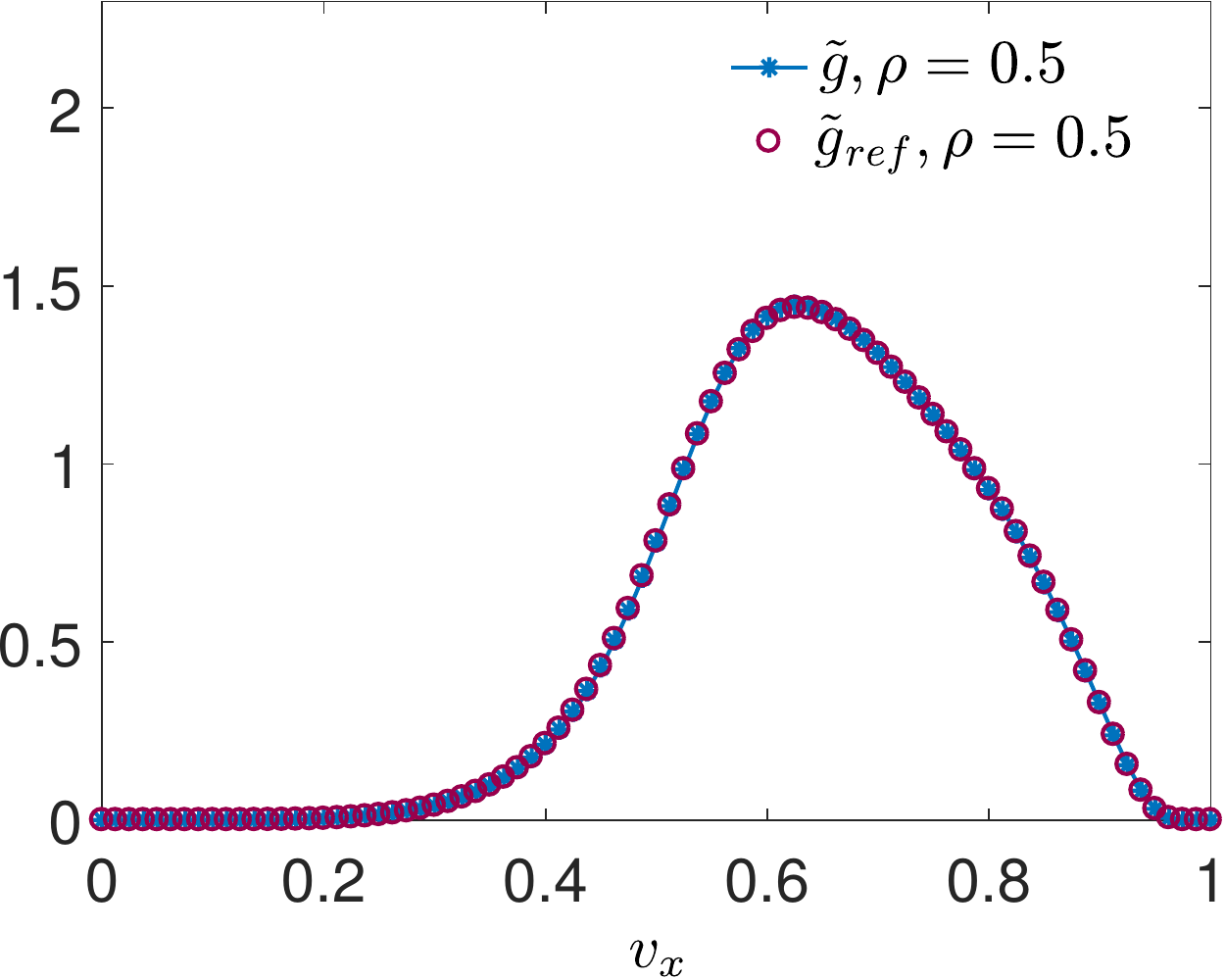}}
\subfigure[$\rho=0.7,\,\sigma^2=10$]{
\includegraphics[scale=0.32]{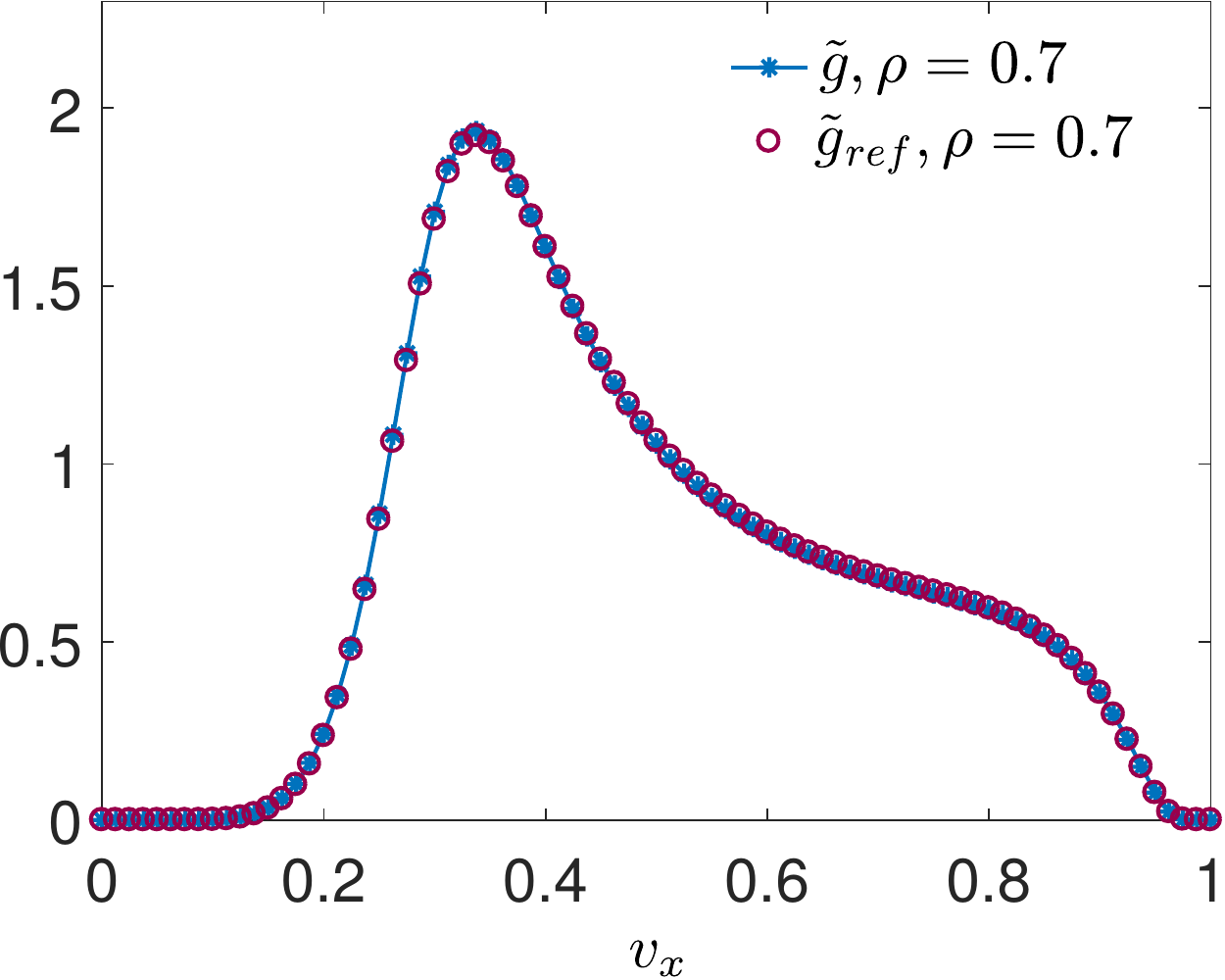}}\\
\subfigure[$\rho=0.3,\,\sigma^2=15$]{
\includegraphics[scale=0.32]{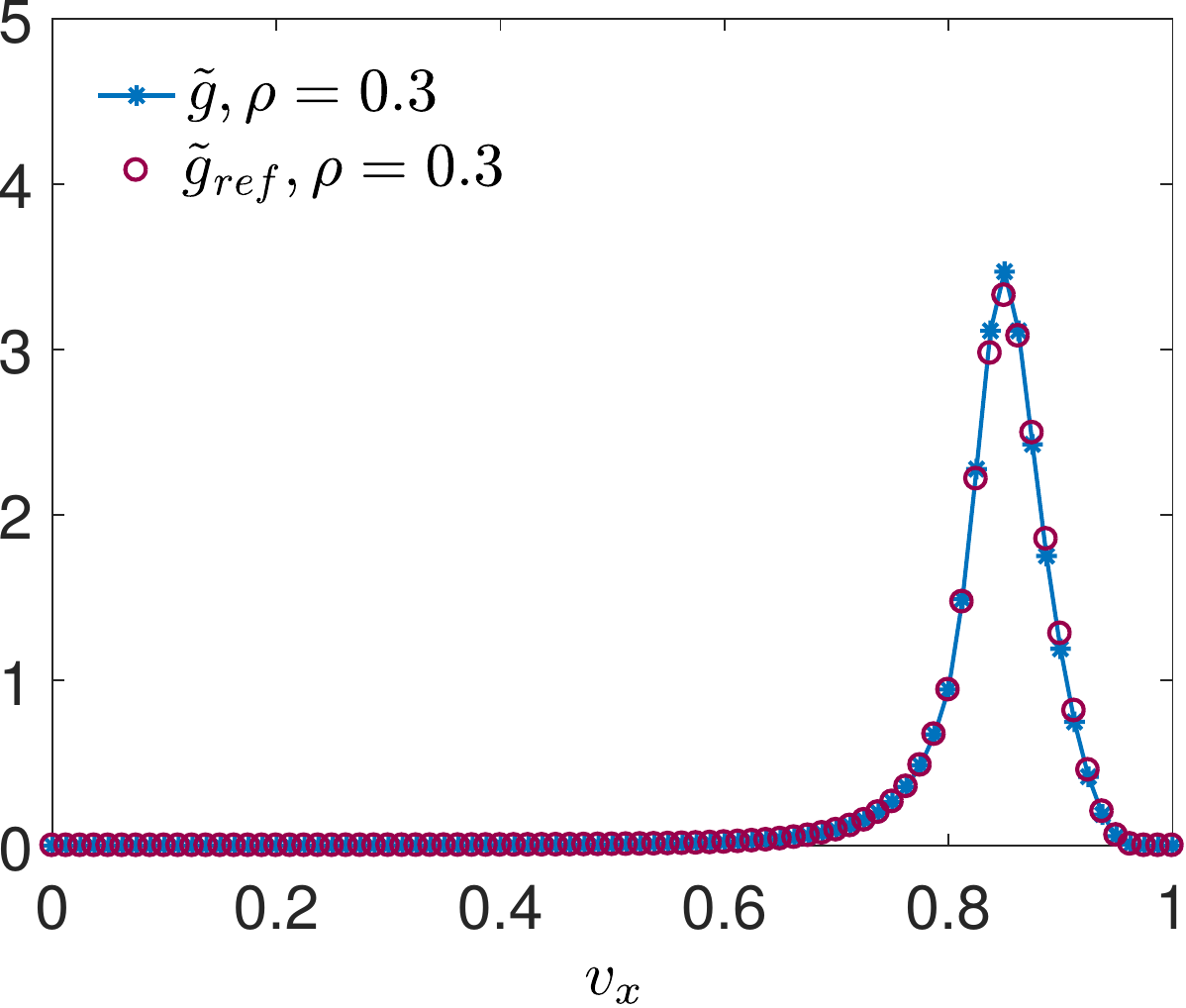}}
\subfigure[$\rho=0.5,\,\sigma^2=15$]{
\includegraphics[scale=0.32]{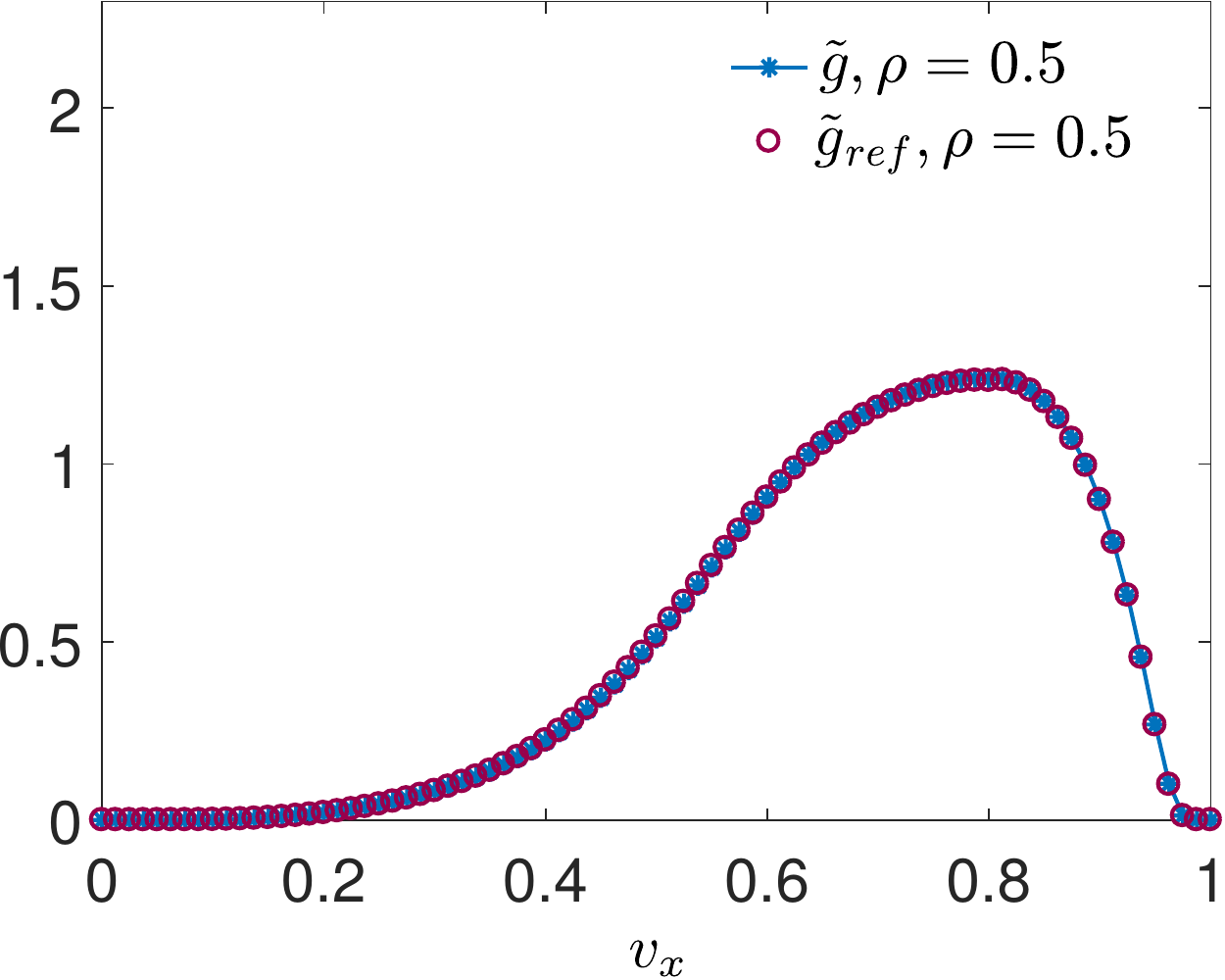}}
\subfigure[$\rho=0.7,\,\sigma^2=15$]{
\includegraphics[scale=0.32]{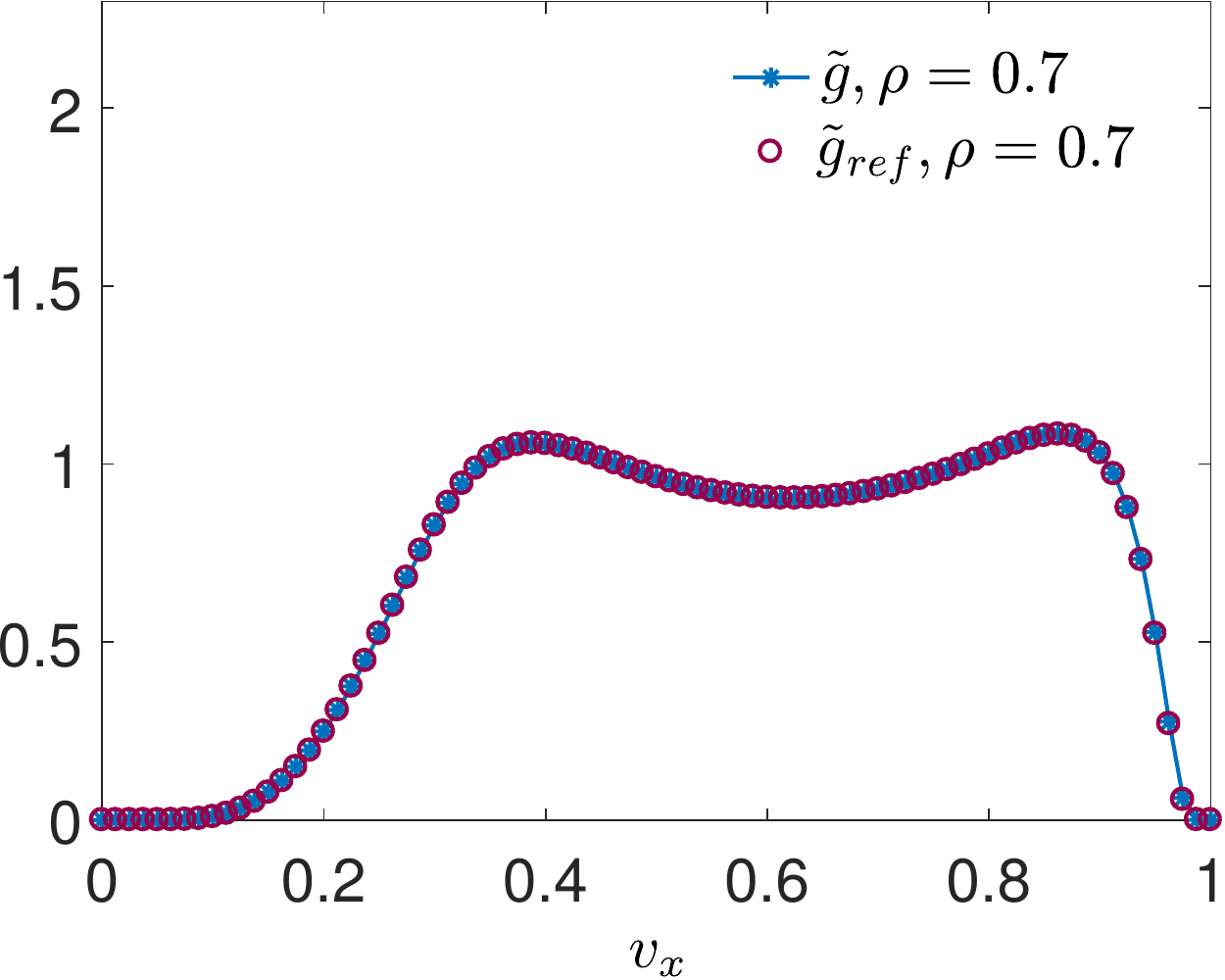}}
\caption{Asymptotic solutions of the Fokker-Planck equation in~\eqref{eq:split.FP} computed with the semi-implicit SP scheme for several values of $\rho$ and $\sigma^2$. The blue line is the solution computed with $N_x=41$ grid points while the red circles represent the reference solution computed with $N_x=321$ grid points.}
\label{fig:stat_prof}
\end{figure}

In order to show the effectiveness of the described scheme we now apply it to the nonlinear Fokker-Planck equation in~\eqref{eq:split.FP} with $W_x=w_x$ and $V_A$, $V_B$ like in~\eqref{eq:VA.VB}. Notice that, since in this illustrative example we are not interested in the coupling with~\eqref{eq:split.Boltz}, we can formally neglect the variables $v_y$, $\theta$, thereby reducing the problem to a one-dimensional equation in the sole variable $v_x$.

In Figure~\ref{fig:rel_errorFP} we present the time evolution, for $\tau\in [0,\,100]$, of the relative $L^1$ error on the numerical solution computed with $N_x=41$ grid points with respect to a reference solution computed with $N_x=321$ grid points. We use the semi-implicit scheme in time with $\Delta{\tau}= \Delta{v_x}/\sigma^2$ and we consider in particular the values $\rho=0.3,\,0.5,\,0.7$ of the vehicle density and $\sigma^2=10,\,15$ of the diffusion coefficient. In Figure~\ref{fig:stat_prof} we sketch the stationary profiles of the kinetic distribution for the same values of $\rho$ and $\sigma^2$.

Finally, in Table~\ref{tab:orderFP} we estimate the order of convergence of the semi-implicit SP scheme for $\rho=0.3,\,0.7$ and $\sigma^2=15$ as $\log_2\frac{e_1(\tau)}{e_2(\tau)}$, where $e_1(\tau)$ is the relative error at time $\tau$ of the solution computed with $N_x=21$ grid points with respect to that computed with $N_x=41$ grid points and, likewise, $e_2(\tau)$ is the relative error at time $\tau$ of the solution computed with $N_x=41$ grid points with respect to that computed with $N_x=81$ grid points. The time step is such that the CFL condition for the positivity of the scheme is satisfied, i.e. $\Delta{ \tau}=O(\Delta{v_x})$ in the semi-implicit case according to Proposition~\ref{prop:SI_dt}.

\begin{table}
\caption{Estimate of the order of convergence of the semi-implicit SP scheme for the Fokker-Planck equation in~\eqref{eq:split.FP} with $\sigma^2=15$ and $\Delta{\tau}=\Delta{v_x}/\sigma^2$.}
\begin{center}
\begin{tabular}{l|r@{.}l|r@{.}l|}
\cline{2-5}
& \multicolumn{2}{|c|}{$\rho=0.3$} & \multicolumn{2}{|c|}{$\rho=0.7$} \\
\hline
\multicolumn{1}{|l|}{$\tau=1$} & $1$ & $7543$ & $1$ & $7794$ \\
\multicolumn{1}{|l|}{$\tau=20$} & $1$ & $9524$ & $1$ & $7821$ \\
\multicolumn{1}{|l|}{$\tau=60$} & $2$ & $2934$ & $1$ & $9282$ \\
\multicolumn{1}{|l|}{$\tau=100$} & $2$ & $3014$ & $1$ & $9283$ \\
\hline                          
\end{tabular}
\end{center}
\label{tab:orderFP}
\end{table}

\subsection{The two-dimensional stochastic model}
We now turn to the numerical solution of the two-dimensional hybrid stochastic model~\eqref{eq:FP-Boltz} by means of the dimensional splitting illustrated at the beginning of Section~\ref{sec:num}. We consider as initial condition the following deterministic distribution:
$$ g(\vv,\,0;\,\theta)=g_0(\vv)=g_{0,x}(v_x)g_{0,y}(v_y), $$
which does not depend on the random parameter $\theta$, with in particular
$$ g_{0,x}(v_x)=\chi_{[0,\,1]}(v_x), \qquad g_{0,y}(v_y)=\frac{1}{2}\chi_{[-1,\,1]}(v_y). $$
Notice that the choice of $g_{0,y}$ implies that we are fixing $\Vy=[-1,\,1]$, i.e. $\e=1$.

As far as the random input $\theta$ is concerned, we choose $\theta\sim\mathcal{U}(-1,\,1)$, hence $\rangeth=[-1,\,1]$ and $h(\theta)=\frac{1}{2}\chi_{[-1,\,1]}(\theta)$, and a desired speed in~\eqref{eq:binary_y} of the form
\begin{equation}
	v_d(\theta)=\bar{v}_d+\lambda P(\rho)\theta,
	\label{eq:vd}
\end{equation}
where $\bar{v}_d\in (-1,\,1)$ and $\lambda>0$ are given constants. For the application of the stochastic collocation method we discretise $\theta$ by means of the first $M>1$ Gauss-Legendre collocation nodes.

Finally, we solve the Boltzmann step~\eqref{eq:split.Boltz} via a direct Monte Carlo method implemented through the Nanbu algorithm~\cite{bobylev2000PRE}. Precisely, we use $N=10^4$ $y$-speeds extracted from the distribution $\tilde{g}$ with a stratified sampling approach~\cite{pareschi2005ESAIMP} and we approximate the collisional equation in~\eqref{eq:split.Boltz} as
$$ g^{n+1}=(1-\mu\rho\Delta{\tau})g^{n+1/2}+\mu\rho\Delta{\tau}Q^+_y(g^{n+1/2}), $$
where $Q^+_y(g):=\frac{1}{1-\beta(u_x)}g('\vv_y,\,\tau;\,\theta)$ is the \emph{gain} term of the collision operator $Q_y$. Observing that $Q^+_y(g)$ is a probability distribution, we see that under the restriction $\mu\rho\Delta{\tau}\leq 1$ the previous equation is a convex combination of the probability distributions $g^{n+1/2}$ and $Q^+_y(g^{n+1/2})$. Hence, in order to obtain new samples distributed according to $g^{n+1}$ we have to sample either from $g^{n+1/2}$, with probability $1-\mu\rho\Delta{\tau}$, or from $Q^+_y(g^{n+1/2})$, with probability $\mu\rho\Delta{\tau}$. As a matter of fact, sampling from $g^{n+1/2}$ corresponds to the event that no $y$-interactions take place during the time step $\Delta{\tau}$ whereas sampling from $Q^+_y(g^{n+1/2})$ corresponds to the event that $y$-interactions occur during the time step $\Delta{\tau}$. In particular, we fix $\mu:=\frac{1}{\rho\Delta{\tau}}$ so that at each time step we do have interactions in the $y$-direction.

\bigskip

\begin{figure}[!t]
\centering
\includegraphics[scale=0.5]{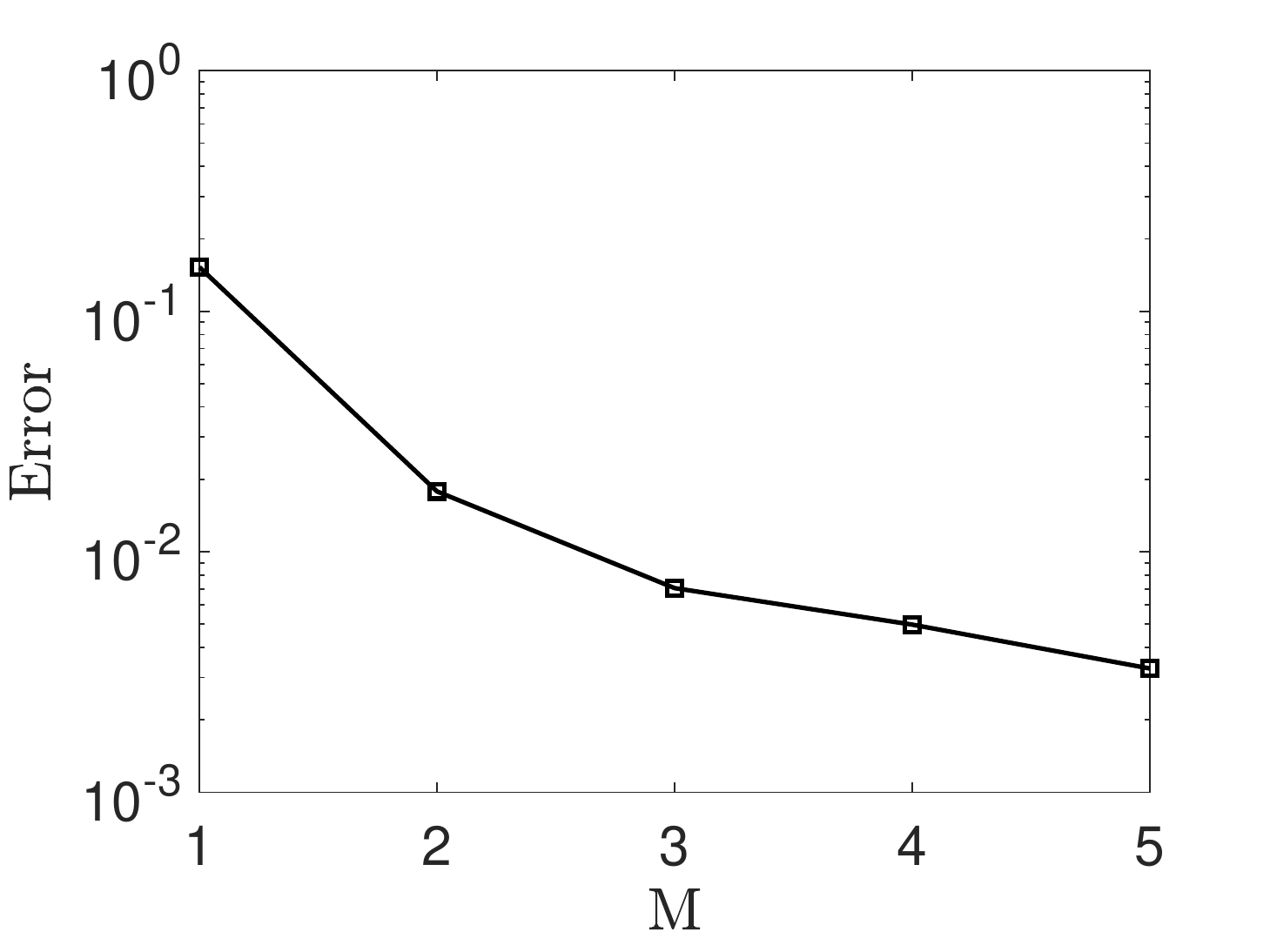}
\includegraphics[scale=0.5]{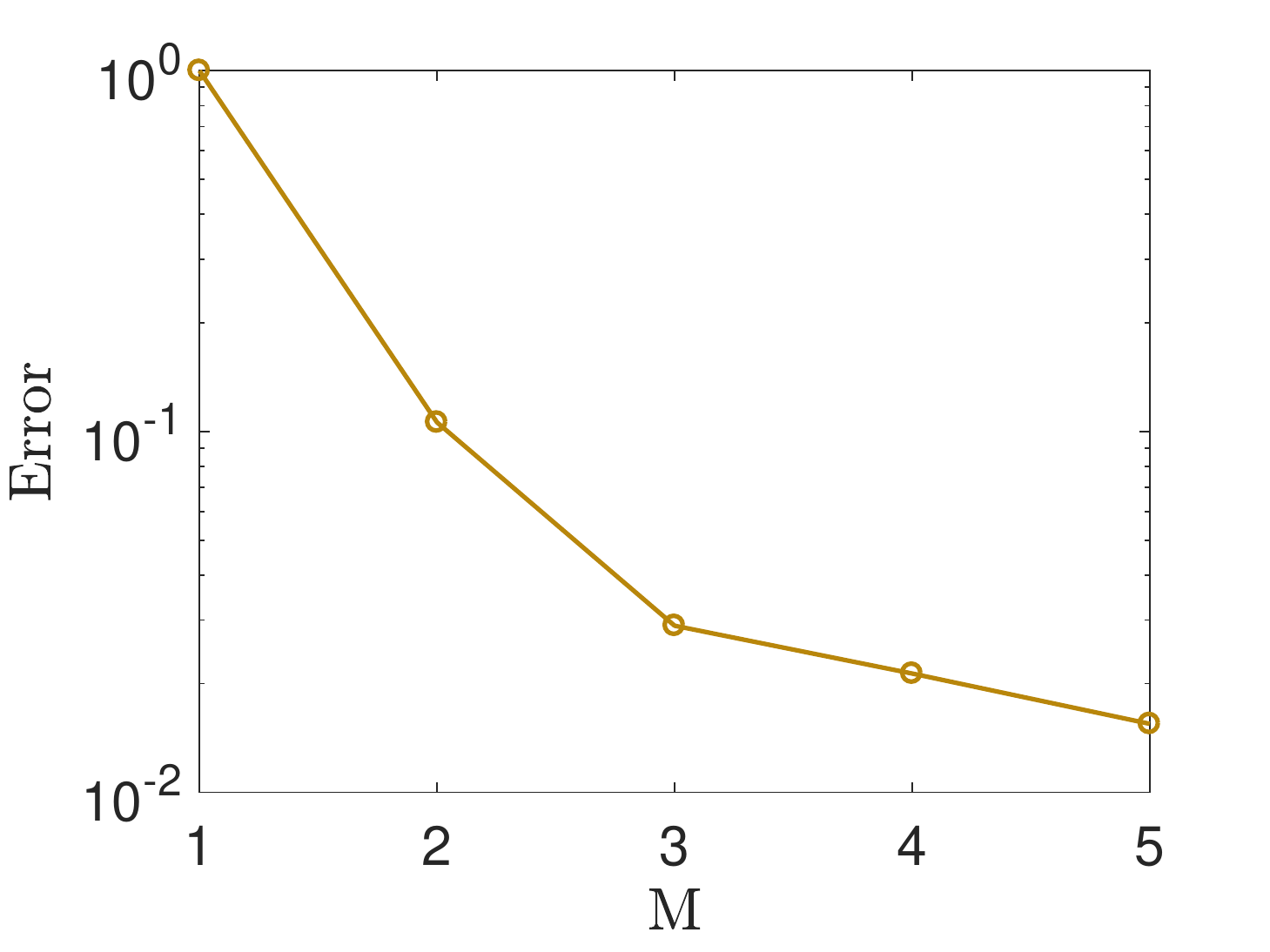}
\caption{Relative $L^1$ error on the $\theta$-expected solution (left) and its $\theta$-variance (right) of~\eqref{eq:FP-Boltz} via the dimensional splitting~\eqref{eq:split.FP}-\eqref{eq:split.Boltz} for an increasing number $M=1,\,\dots,\,5$ of collocation nodes in $\rangeth$. The reference numerical solution is computed with $M=50$ collocation nodes. In~\eqref{eq:vd} we have fixed $\bar{v}_d=0$ and $\lambda=10^{-1}$.}
\label{fig:expcon}
\end{figure}

In Figure~\ref{fig:expcon} we show the relative $L^1$ error on the expected solution $\bar{g}$ and its variance $\Var_\theta(g)$, cf.~\eqref{eq:fbar.var}, of the full two-dimensional problem for an increasing number $M=1,\,\dots,\,5$ of collocation nodes in $\rangeth$, computed with respect to a reference numerical solution on $M=50$ nodes. The variable $v_x$ has been discretised with $N_x=101$ grid points in $\Vx$, the variable $v_y$ with $N_y=41$ grid points in $\Vy$ and a time step $\Delta{\tau}=O(\Delta{v_x})$ has been chosen for the semi-implicit SP scheme, cf. Proposition~\ref{prop:SI_dt}, with final time $T=100$. We point out that, due to the stratified sampling approach, the curves plotted in Figure~\ref{fig:expcon} are actually averages computed out of $10^2$ estimates of the relative error.

\begin{figure}[!t]
\centering
\subfigure[$\rho=0.3, \tau=1$]{\includegraphics[scale=0.3]{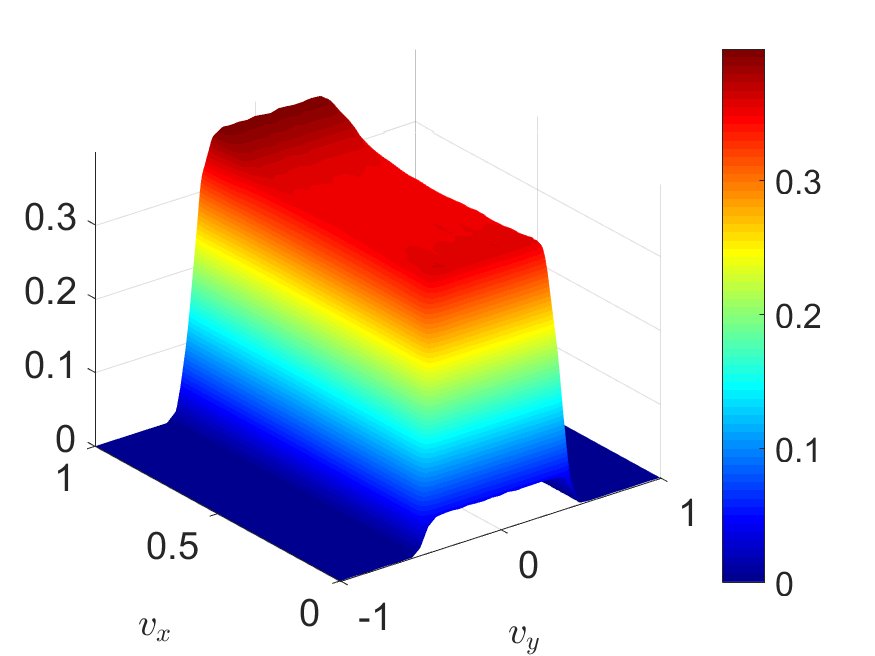}}
\subfigure[$\rho=0.3, \tau=5$]{\includegraphics[scale=0.3]{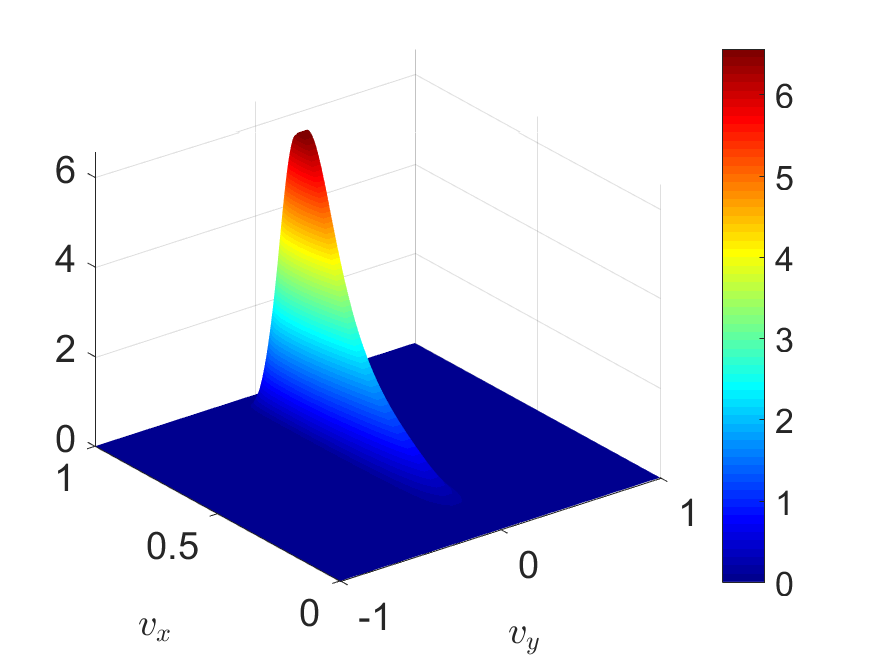}}
\subfigure[$\rho=0.3, \tau=50$]{\includegraphics[scale=0.3]{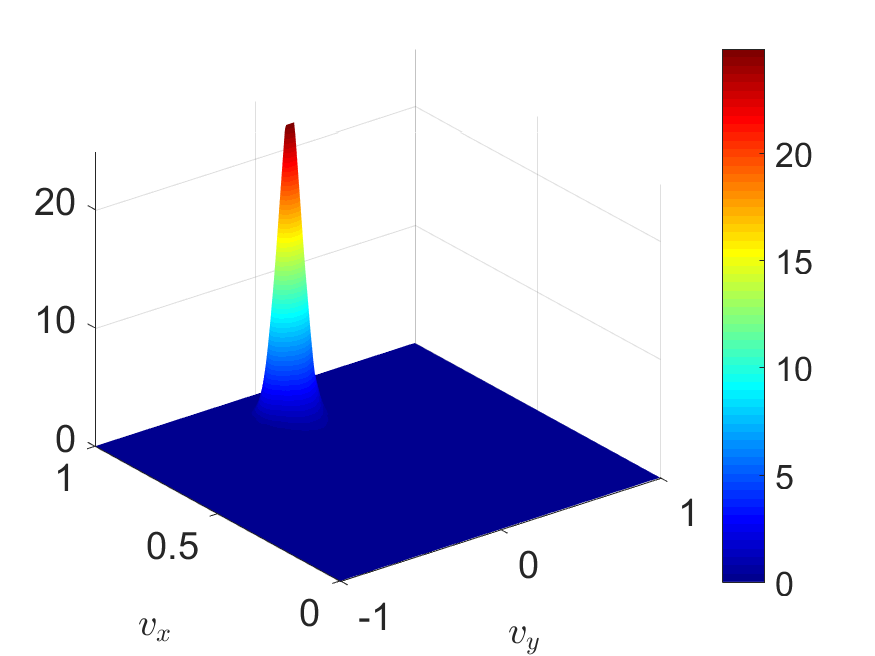}}	\\
\subfigure[$\rho=0.6, \tau=1$]{\includegraphics[scale=0.3]{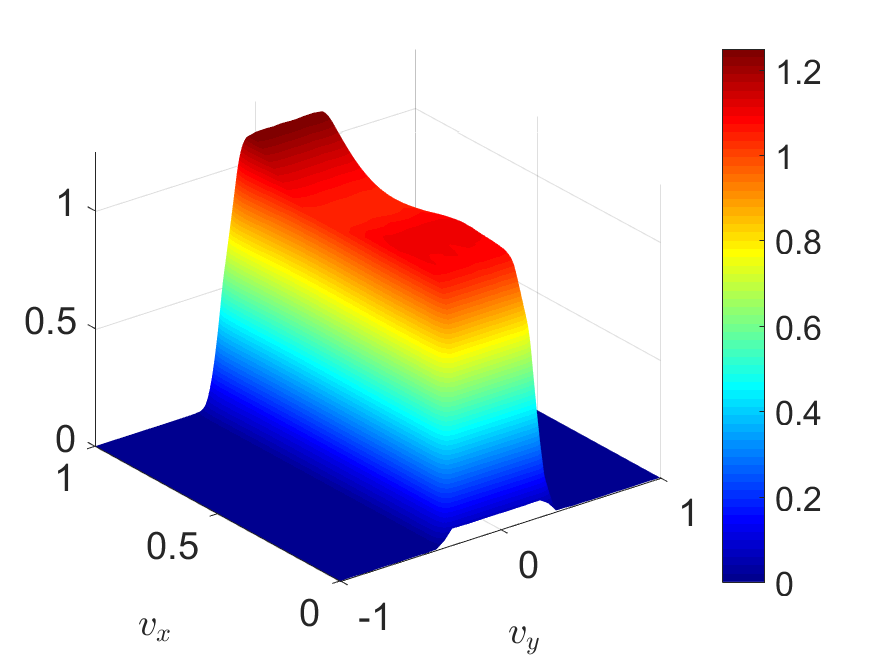}}
\subfigure[$\rho=0.6, \tau=5$]{\includegraphics[scale=0.3]{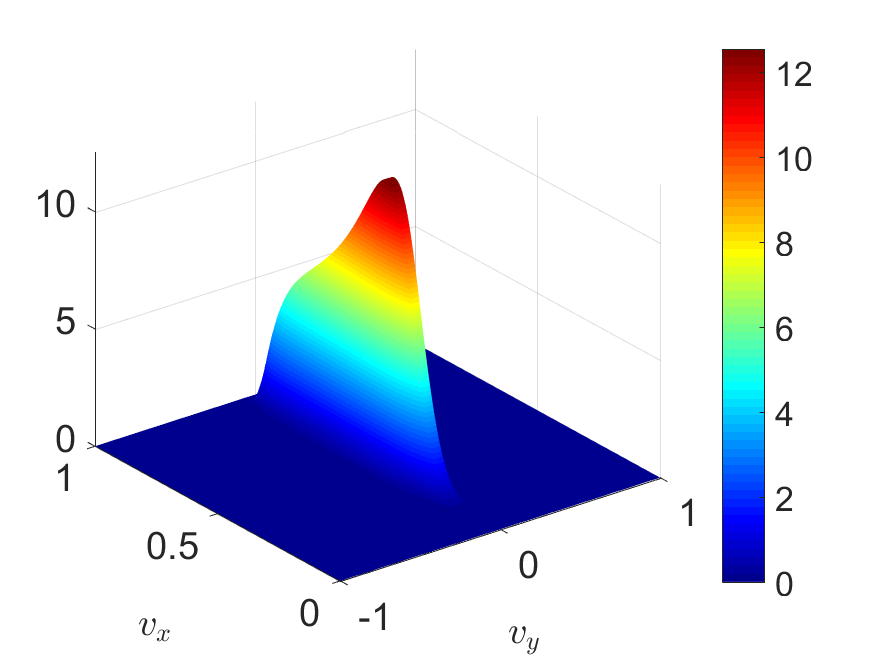}}
\subfigure[$\rho=0.6, \tau=50$]{\includegraphics[scale=0.3]{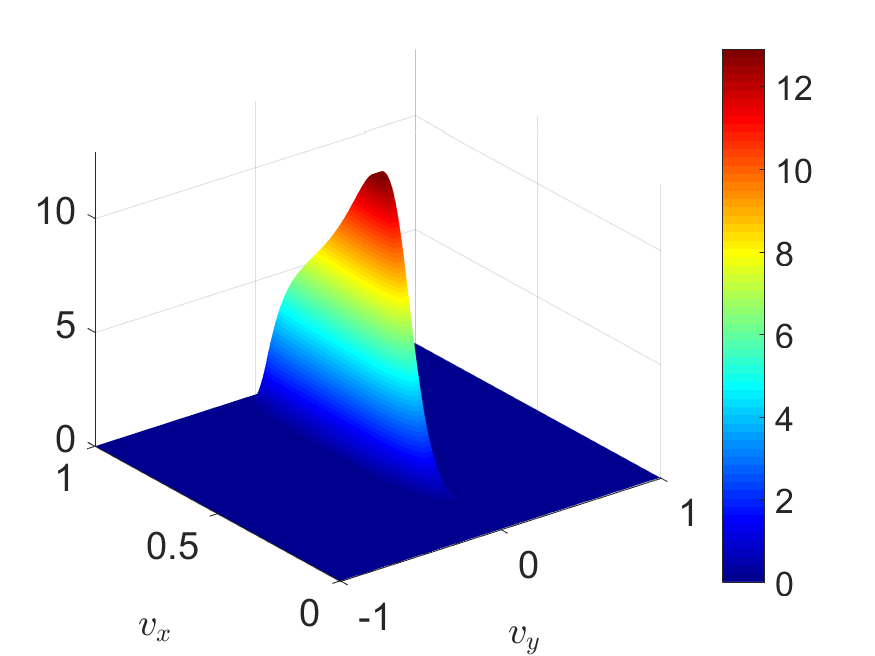}}\\
\subfigure[$\rho=0.9, \tau=1$]{\includegraphics[scale=0.3]{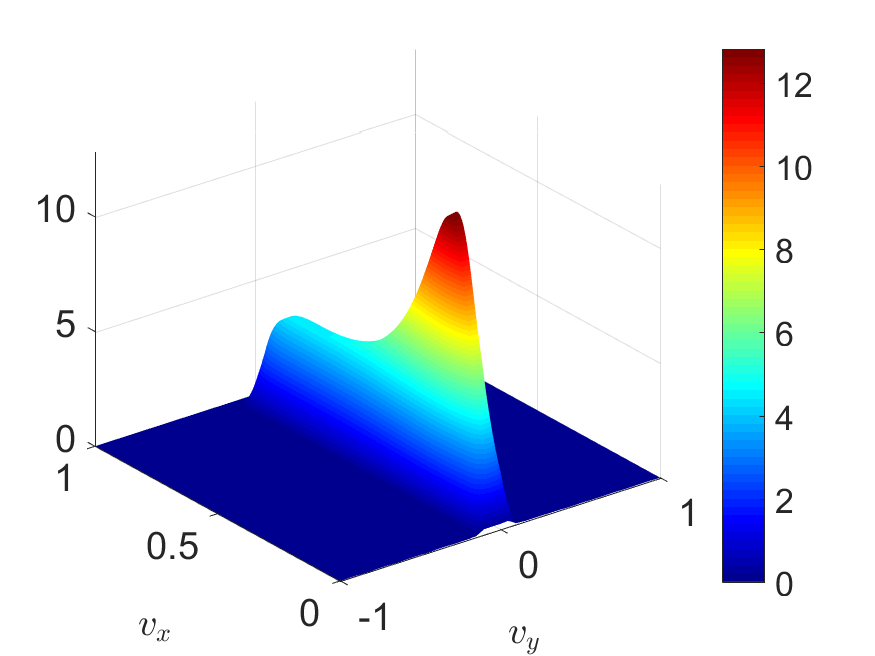}}
\subfigure[$\rho=0.9, \tau=5$]{\includegraphics[scale=0.3]{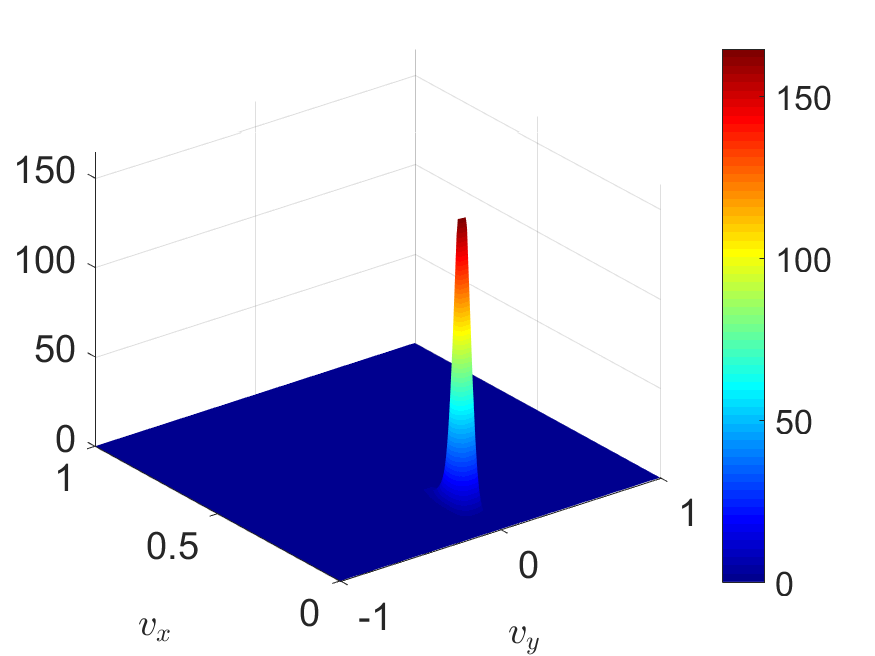}}
\subfigure[$\rho=0.9, \tau=50$]{\includegraphics[scale=0.3]{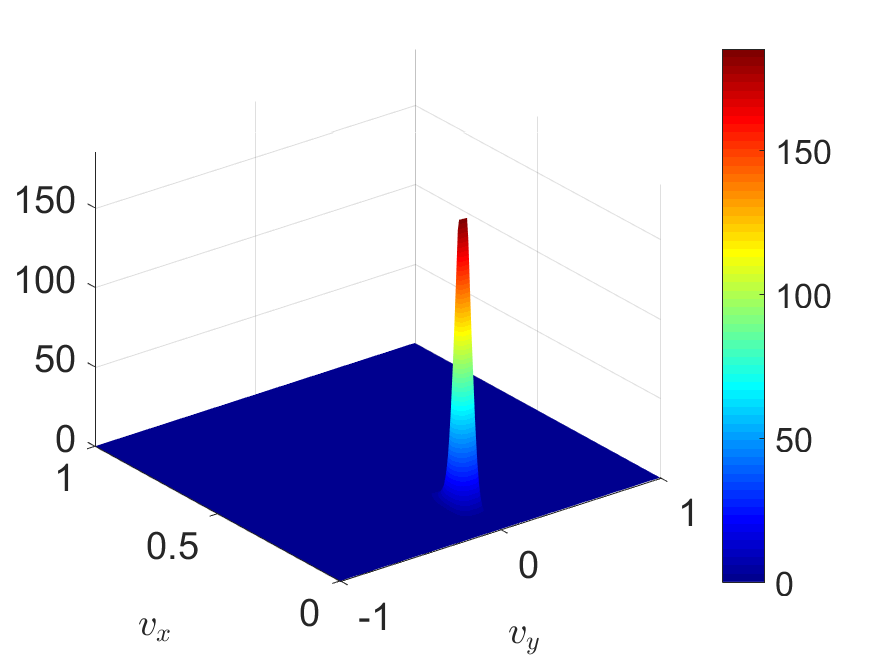}}
\caption{Evolution of the $\theta$-expected solution of~\eqref{eq:FP-Boltz} for different values of the traffic density, in particular $\rho=0.3$ (top row), $\rho=0.6$ (middle row) and $\rho=0.9$ (bottom row).}
\label{fig:exp_2D}
\end{figure}

\begin{figure}[!t]
\centering
\subfigure[$\rho=0.3, \tau=1$]{\includegraphics[scale=0.3]{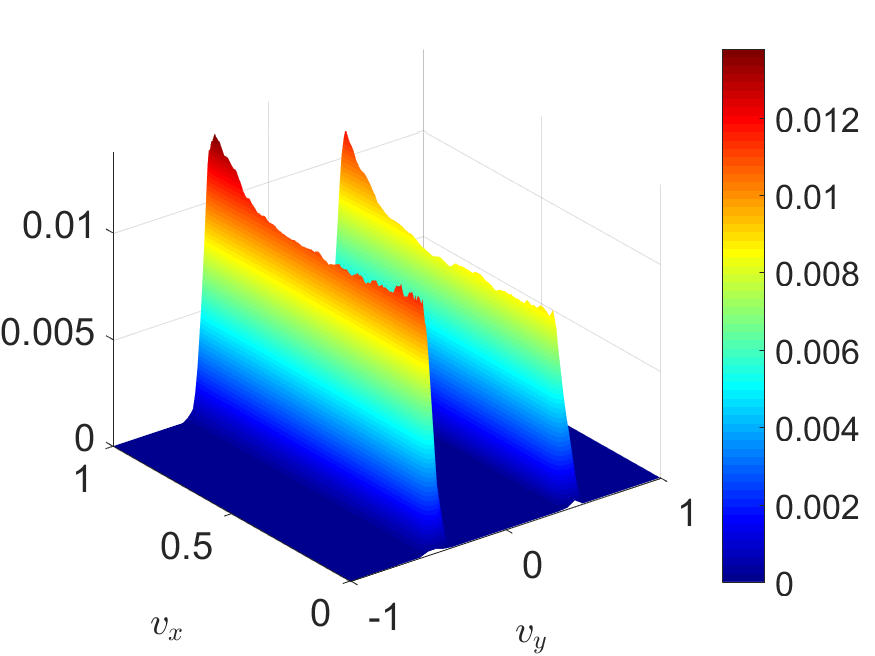}}
\subfigure[$\rho=0.3, \tau=5$]{\includegraphics[scale=0.3]{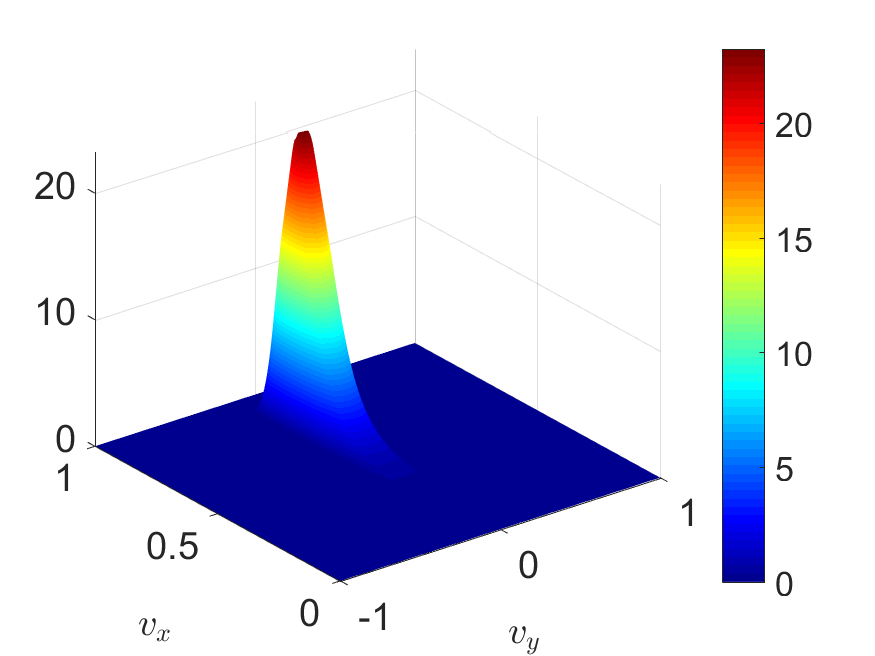}}
\subfigure[$\rho=0.3, \tau=50$]{\includegraphics[scale=0.3]{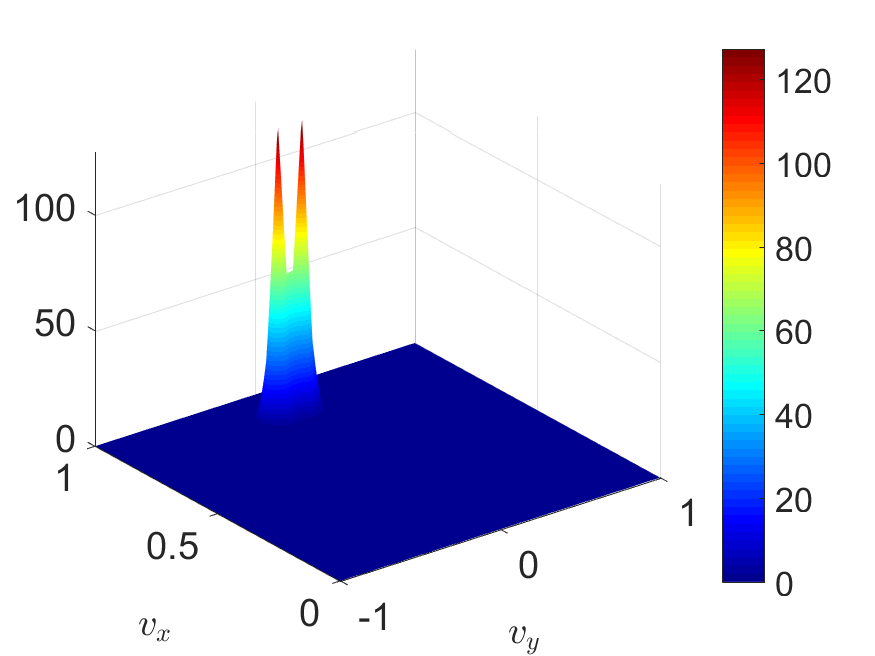}}	\\
\subfigure[$\rho=0.6, \tau=1$]{\includegraphics[scale=0.3]{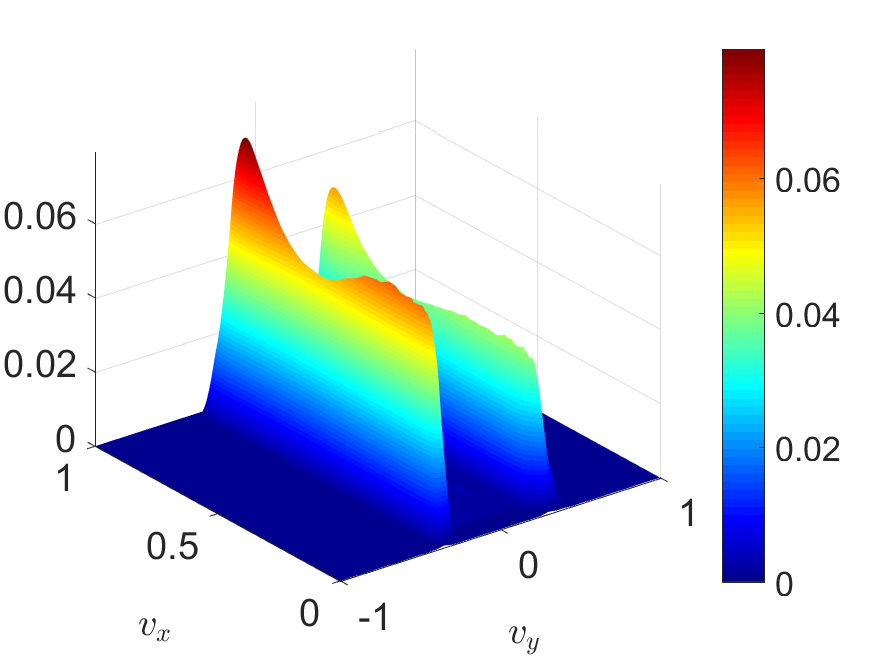}}
\subfigure[$\rho=0.6, \tau=5$]{\includegraphics[scale=0.3]{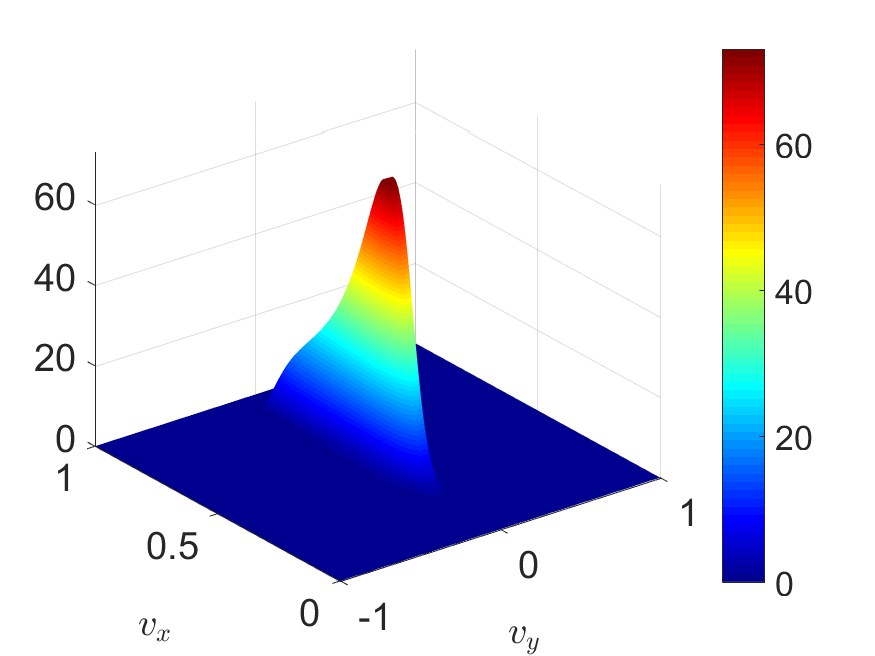}}
\subfigure[$\rho=0.6, \tau=50$]{\includegraphics[scale=0.3]{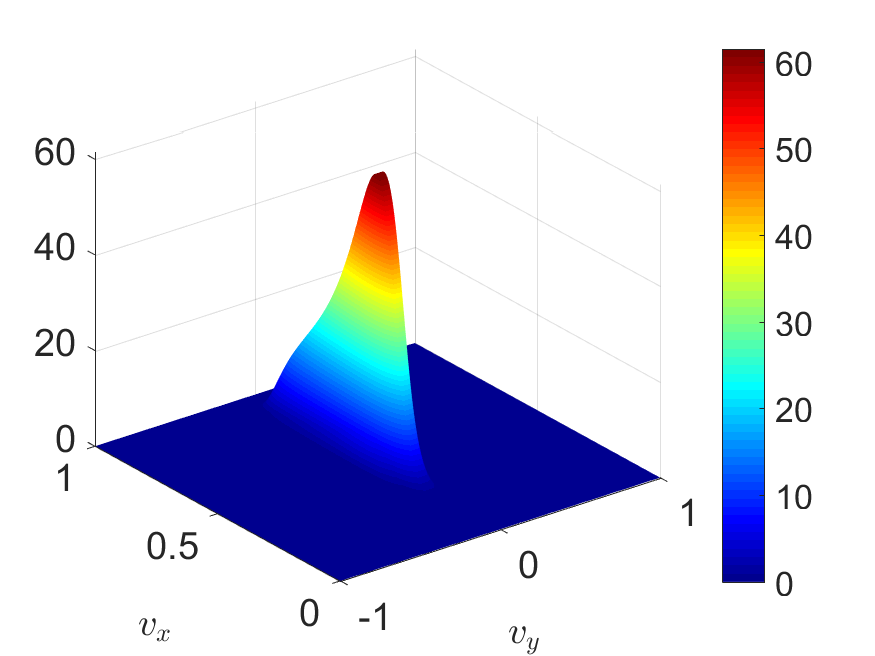}}\\
\subfigure[$\rho=0.9, \tau=1$]{\includegraphics[scale=0.3]{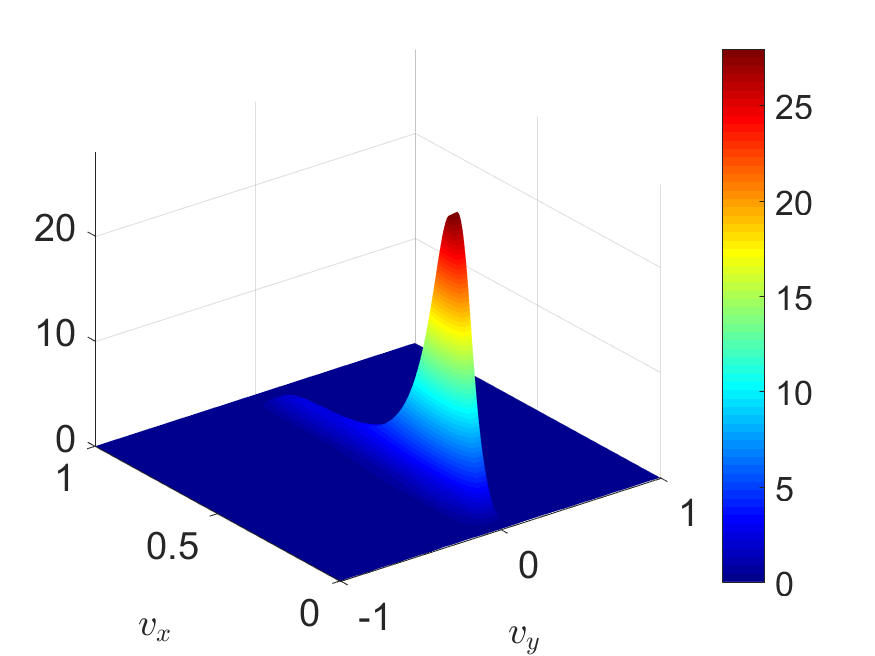}}
\subfigure[$\rho=0.9, \tau=5$]{\includegraphics[scale=0.3]{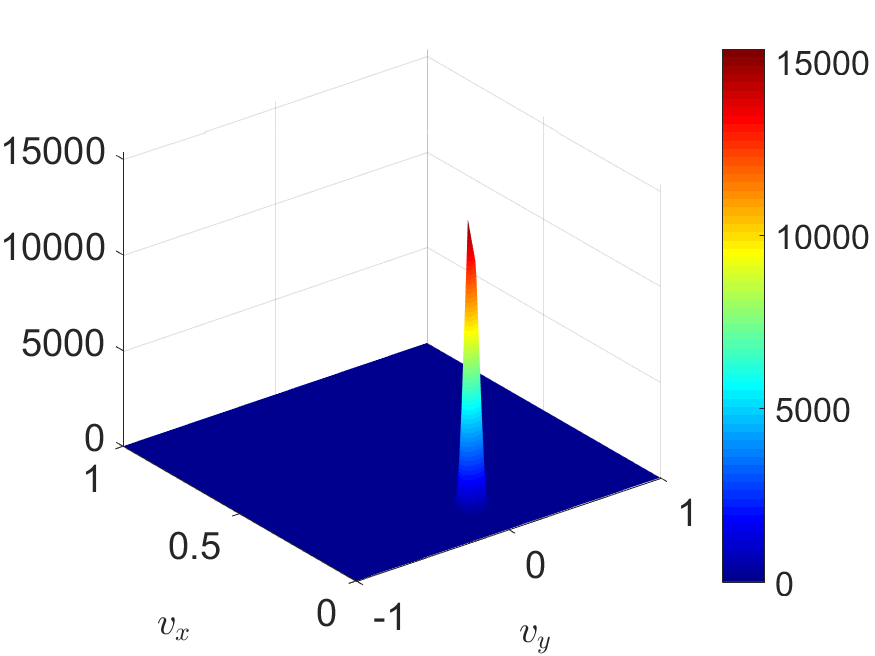}}
\subfigure[$\rho=0.9, \tau=50$]{\includegraphics[scale=0.3]{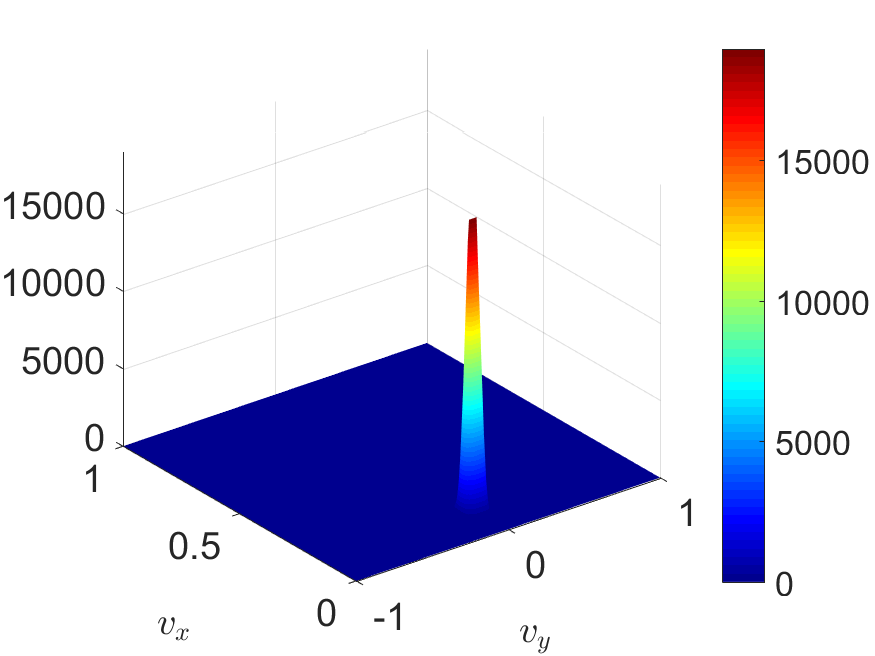}}
\caption{Evolution of the $\theta$-variance of the expected solution to~\eqref{eq:FP-Boltz} for different values of the traffic density, in particular $\rho=0.3$ (top row), $\rho=0.6$ (middle row) and $\rho=0.9$ (bottom row).}
\label{fig:var_2D}
\end{figure}

In Figures~\ref{fig:exp_2D},~\ref{fig:var_2D} we show instead the time evolution of the expected solution $\bar{g}$ and of its variance $\Var_\theta(g)$, respectively, cf. again~\eqref{eq:fbar.var}, for the traffic densities $\rho=0.3,\,0.6,\,0.9$. In particular, the graphs of $\Var_\theta(g)$ highlight the regions of the space of the microscopic states $\VV=\Vx\times\Vy$ where the statistical variability of the expected solution $\bar{g}$ is higher due to the uncertainty in $v_d(\theta)$.

\subsection{Two-dimensional speed diagrams of traffic}
\label{sec:funddiag}
A usual benchmark for validating a kinetic traffic model consists in checking if the theoretical speed diagrams arising from the asymptotic kinetic distributions reproduce the features typically observed in the empirical speed diagrams.

Speed diagrams express the average speed of the vehicles as a function of the vehicle density at equilibrium and in spatially homogeneous conditions. In our case, we compute the mean speeds at equilibrium from the asymptotic solution of~\eqref{eq:FP-Boltz} as
\begin{align*}
	u^\infty_x &= \int_{\VV}v_xg^\infty(\vv;\,\theta)\,d\vv=\int_0^1v_xg^\infty_x(v_x)\,dv_x, \\
	u^\infty_y(\theta) &= \int_{\VV}v_yg^\infty(\vv;\,\theta)\,d\vv=\int_{-\e}^\e v_yg^\infty_y(v_y;\,\theta)\,dv_y.
\end{align*}
Notice that only the asymptotic $y$-mean speed is actually uncertain, because the $y$-dynamics~\eqref{eq:binary_y} contain the random input $\theta$. The asymptotic $x$-mean speed is not, because the $x$-dynamics~\eqref{eq:binary_x-v}-\eqref{eq:binary_x-w} do not contain any random input nor any explicit coupling with the $y$-dynamics. By further averaging $u^\infty_y$ with respect to the uncertainty in $\theta$ we get:
$$ \bar{u}^\infty_y:=\int_{\rangeth}u^\infty_y(\theta)h(\theta)\,d\theta, $$
while obviously it results $\bar{u}^\infty_x=u^\infty_x$. Clearly, these values depend on the density $\rho\in [0,\,1]$ fixed in~\eqref{eq:LD}, hence in~\eqref{eq:LDcal}, and in~\eqref{eq:FP-Boltz}. Finally, we plot the mappings $\rho\mapsto\bar{u}^\infty_y$, $\rho\mapsto u^\infty_x$ and compare them with the empirical ones obtained from a dataset described below, cf. Section~\ref{sec:dataset}.

In order to reproduce the scattering of the experimental data normally seen in the measured speed diagrams, we compute the following indicator of maximum dispersion of the $y$-energy:
$$ I_y:=\bar{E}^\infty_y+\sqrt{\Var_\theta(E^\infty_y)}, $$
where $\bar{E}^\infty_y$ denotes the $\theta$-expected value of the asymptotic energy in the $y$-direction and $\Var_\theta(E^\infty_y)$ its $\theta$-variance, then we plot the mappings $\rho\mapsto\bar{u}^\infty_y\pm\sqrt{I_y}$. Notice that $I_y$ represents the deviation from the expected $y$-energy of the model, hence its square root is dimensionally comparable to a speed.

\begin{figure}[t!]
\centering
\subfigure[]{\includegraphics[width=0.49\textwidth]{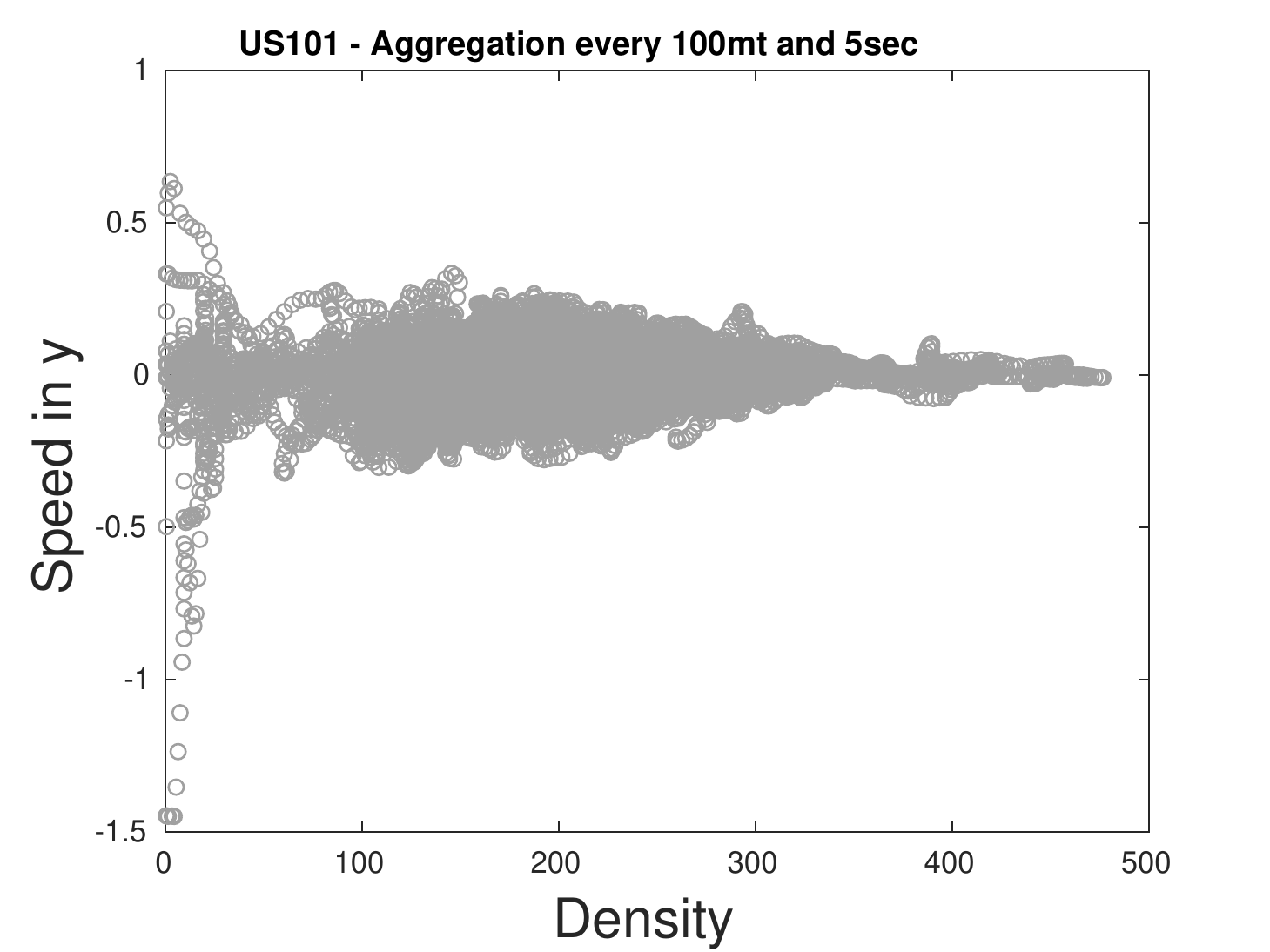}}
\subfigure[]{\includegraphics[width=0.49\textwidth]{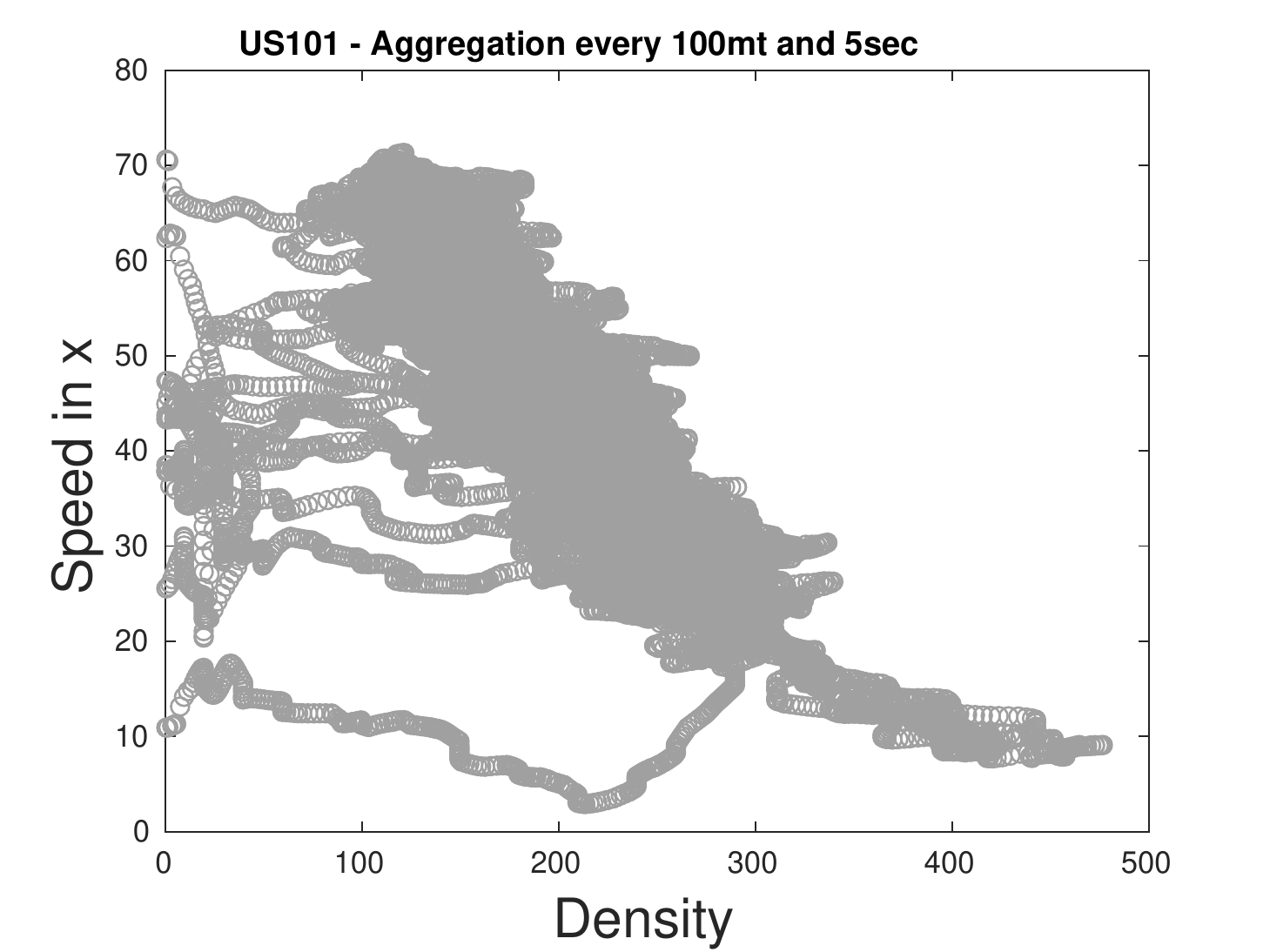}}
\caption{Speed-density diagrams in the $y$-direction (a) and $x$-direction (b). The grey circles are the experimental speeds obtained from the U.S. dataset described in Section~\ref{sec:dataset}.}
\label{fig:expdiagrams}
\end{figure}

\subsubsection{Experimental dataset}
\label{sec:dataset}
We consider a set of experimental data recorded on a section on the southbound direction of the U.S. Highway $101$ (known as Hollywood Freeway) in Los Angeles, California. The data are part of the Federal Highway Administration's (FHWA) Next Generation Simulation (NGSIM) project~\cite{NGSIM}. They consist of the two-dimensional vehicle trajectories collected between 7.50 am and 8.15 am on 15th June 2005 using $8$ video cameras with resolution of $10$ frames per second. The road section is approximately $640~\text{m}$ in length with five main lanes plus an auxiliary lane in the corridor between an incoming and an outgoing ramp. However, we only consider the stretch as if there were no ramps.

The microscopic velocities of the vehicles are recovered out of their microscopic positions in consecutive frames. From the microscopic data, the macroscopic quantities in each direction of the flow can be computed as explained e.g. in~\cite{hoogendoorn2007LN,maerivoet2005TECHREP}. The aggregation of the data is made with respect to time ($5~\text{s}$) and distance ($100~\text{m}$).

In Figure~\ref{fig:expdiagrams} we observe that the order of magnitude of the recorded mean speed in the $y$-direction is much smaller than that in the $x$-direction, which indeed justifies the formulation of a hybrid kinetic model to clearly separate the two speed scales. We stress that, to our knowledge, this is one of the first times that speed diagrams are recorded also for lateral displacements of the cars across the lanes.

\begin{figure}[t!]
\centering
\subfigure[]{\includegraphics[width=0.49\textwidth]{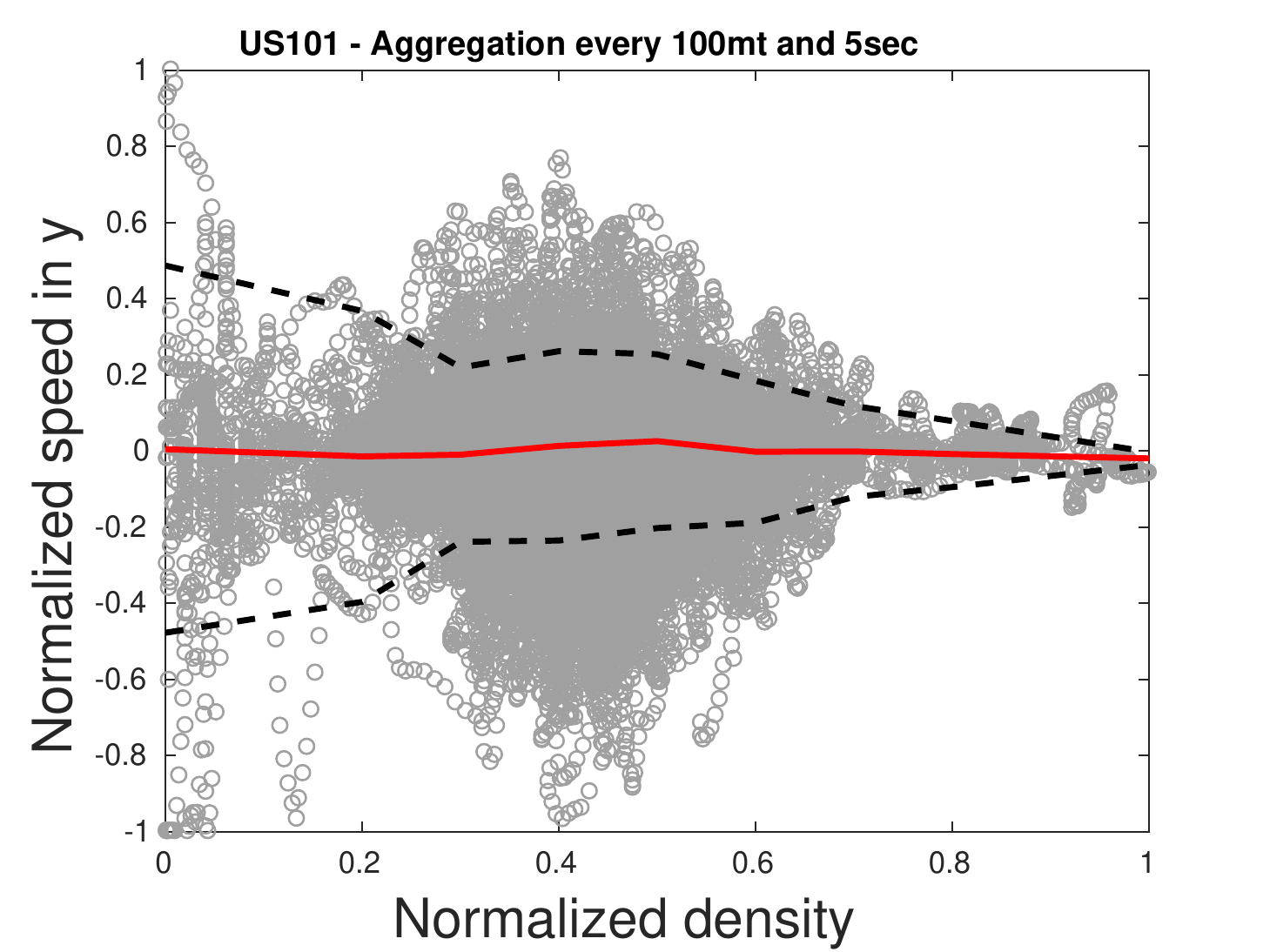}}
\subfigure[]{\includegraphics[width=0.49\textwidth]{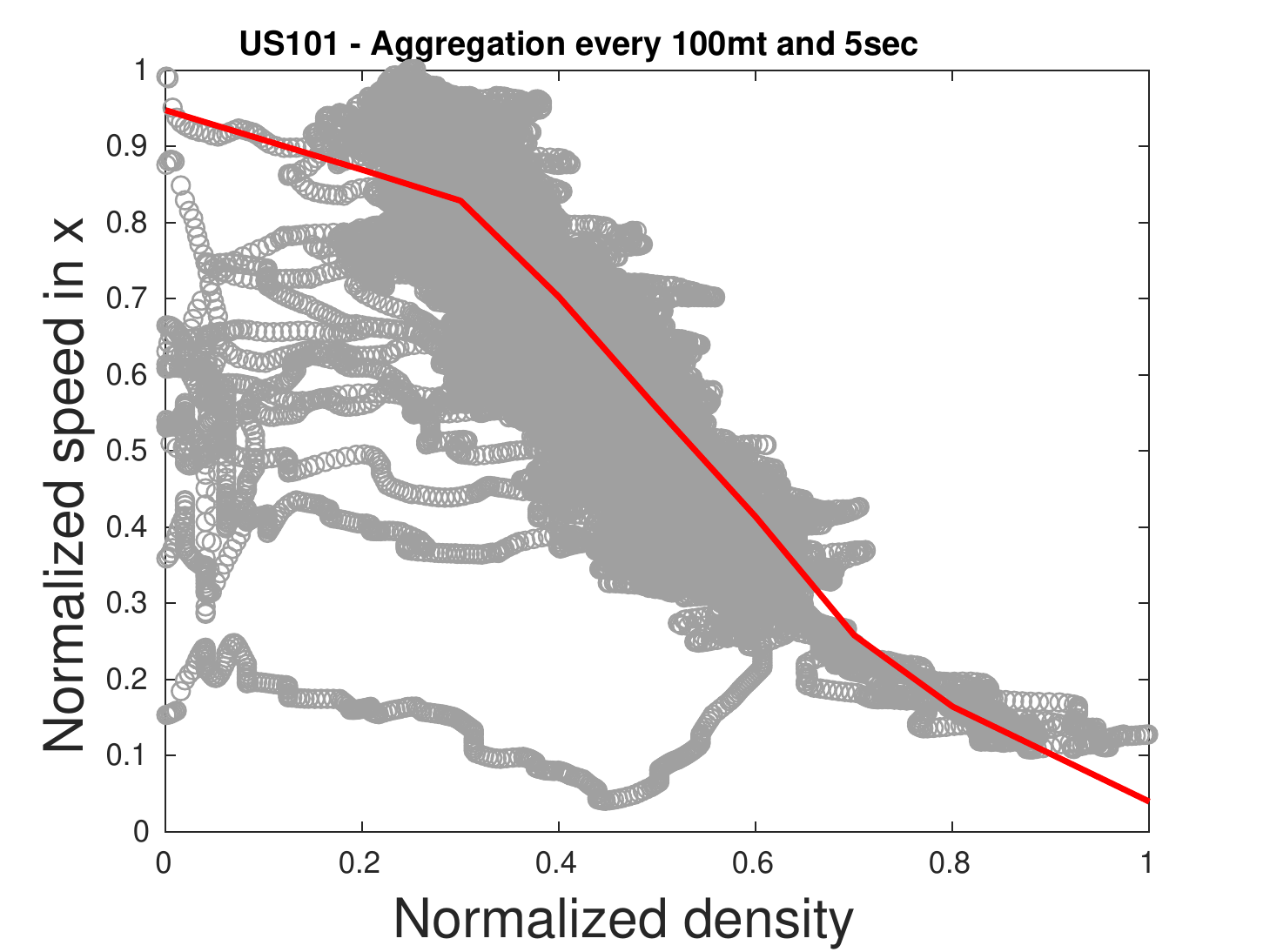}}
\caption{Theoretical speed-density diagrams (red, black lines) in the $y$-direction (a) and the $x$-direction (b) on top of the normalised empirical ones (grey circles). The red solid lines are the $\theta$-expected asymptotic mean speeds in the two directions of the flow while the black dashed lines in (a) are the deviations from the mean speed $\bar{u}^\infty_y$ given by the energy-based estimator $I_y$. In~\eqref{eq:P} we have considered $\delta=1$ while in~\eqref{eq:vd} we have set $\bar{v}_d=-0.0109$ and $\lambda=\frac{1}{2}$.}
\label{fig:diagrams}
\end{figure}

\subsubsection{Theoretical speed diagrams}
\label{sec:fund_teo}
In Figure~\ref{fig:diagrams} we show the theoretical speed-density diagrams, computed as discussed at the beginning of Section~\ref{sec:funddiag}, in both the $y$-direction (a) and the $x$-direction (b). For straightforward comparison, we place them on top of the empirical data (grey circles), which here are duly normalised with respect to the maximum density and the maximum speed in either direction of the flow for consistency with the dimensionless results of the mathematical model.

The red solid lines are the graphs of the mappings $\rho\mapsto\bar{u}^\infty_y$ in Figure~\ref{fig:diagrams}a and $\rho\mapsto u^\infty_x$ in Figure~\ref{fig:diagrams}b. The black dashed lines in Figure~\ref{fig:diagrams}a are instead the graphs of the mappings $\rho\mapsto\bar{u}^\infty_y\pm\sqrt{I_y}$. On the whole, we see that the theoretical results reproduce quite well the measurements. In particular, we notice that the theoretical speed diagram in the $y$-direction is constant at all densities around the value $\bar{v}_d$ in~\eqref{eq:vd}, that here we set to $\bar{v}_d=-0.0109$ (estimated from the data). In contrast, the theoretical speed diagram in the $x$-direction shows the typical decreasing trend towards zero at high density (congested traffic phase) after a nearly constant trend for low density (free traffic phase).

As far as the data scattering is concerned, we observe that the energy-based confidence interval around $\bar{u}^\infty_y$ estimated by means of $I_y$, cf. Figure~\ref{fig:diagrams}a, encompasses most of the empirical speed values, thus suggesting that the data dispersion in the $y$-direction can be indeed explained in terms of the stochastic variability due to $\theta$ in the microscopic dynamics across the lanes, cf.~\eqref{eq:binary_y},~\eqref{eq:vd}. Again, we stress that, to our knowledge, this is one of the first times that theoretical speed diagrams due to lane changes are studied and explained by a mathematical model. Since we have not included any source of uncertainty in the $x$-dynamics~\eqref{eq:binary_x-v}-\eqref{eq:binary_x-w}, we cannot reproduce a similar estimate of the data dispersion in the theoretical $x$-speed diagram. However, for the sake of completeness, we mention that other works offer alternative explanations for the data dispersion in the $x$-direction which do not appeal to uncertain parameters nor UQ, see e.g.~\cite{fan2014NHM,puppo2016CMS,visconti2017MMS}.

Thanks to the results of Section~\ref{sec:asympt.distr}, we can also compute analytically the theoretical curves appearing in the $y$-speed diagram of Figure~\ref{fig:diagrams}a. In fact from $g^\infty_y(v_y;\,\theta)=\delta_{v_d(\theta)}(v_y)$ we have $u^\infty_y(\theta)=v_d(\theta)$, whence using~\eqref{eq:vd} with $\theta\sim\mathcal{U}(-1,\,1)$ we obtain
$$ \bar{u}^\infty_y=\frac{1}{2}\int_{-1}^{1}v_d(\theta)\,d\theta=\bar{v}_d, $$
which is the equation of the red line in Figure~\ref{fig:diagrams}a. Moreover, since $E^\infty_y(\theta)=v_d^2(\theta)$ we compute:
\begin{align*}
	\bar{E}^\infty_y(\theta) &= \frac{1}{2}\int_{-1}^{1}v_d^2(\theta)\,d\theta=\bar{v}_d^2+\frac{1}{3}\lambda^2P^2(\rho) \\
	\Var_\theta(E^\infty_y) &= \frac{1}{2}\int_{-1}^{1}v_d^4(\theta)\,d\theta-\left(\bar{E}^\infty_y\right)^2=
		4\left(\frac{1}{3}\bar{v}_d^2+\frac{1}{45}\lambda^2P^2(\rho)\right)\lambda^2P^2(\rho).
\end{align*}
In particular, for $\bar{v}_d=0$ (nearly the value used for the simulated diagram of Figure~\ref{fig:diagrams}a) this gives $I_y=\left(\frac{1}{3}+\frac{2}{\sqrt{45}}\right)\lambda^2P^2(\rho)$, thus the curves of the energy-based confidence interval in the $y$-speed diagram are
$$ \bar{u}^\infty_y\pm\sqrt{I_y}=\pm\sqrt{\frac{1}{3}+\frac{2}{\sqrt{45}}}\lambda P(\rho)
	\approx \pm 0.4(1-\rho) $$
for the expression~\eqref{eq:P} of $P(\rho)$ with the values of $\lambda$, $\delta$ in Figure~\ref{fig:diagrams}a.

\section{Conclusion}
\label{sec:conclusion}
In this work we have introduced a hybrid stochastic kinetic model of two-dimensional traffic dynamics, which takes into account speed changes both along and across road lanes in consequence of vehicle interactions and lane changes, respectively. Starting from a Boltzmann-type description based on suitable microscopic dynamics, we have derived a hybrid Fokker-Planck-Boltzmann equation in the quasi-invariant interaction limit assuming that lane changes, described by a linear collision operator, are much less frequent than speed variations along the lanes, described by a nonlinear Fokker-Planck operator. In particular, we have suggested that speed variations due to lane changes can be modelled at the microscopic level simply as a relaxation process towards a desired lateral speed, which however is not known deterministically. This introduces an intrinsic uncertainty in the kinetic equation, which proves to be essential for reproducing theoretically not only the average macroscopic trends observed in reality but also the scattering of the experimental data typical of the empirical fundamental diagrams of traffic.

Besides the result just mentioned, the main methodological contributions of this work are the following:
\begin{inparaenum}[($i$)]
\item we have proposed a formal asymptotic procedure to derive hybrid kinetic models including uncertain parameters, which can be applied to multivariate microscopic dynamics when some of them occur at a much lower rate than others. The advantage is that the most frequent dynamics turn out to be modelled by Fokker-Planck-type differential operators replacing the original Boltzmann-type collision operators, while the latter remain to model only the less frequent dynamics;
\item we have proposed a numerical study of the general hybrid stochastic kinetic equation by means of an extension of structure preserving methods existing in the literature to fully nonlinear Fokker-Planck equations combined with direct Monte Carlo methods, stratified sampling and UQ collocation methods to quantify the uncertainty intrinsic in the stochastic kinetic equation.
\end{inparaenum}

Further amplifications of the present work may include a systematic study of the numerical method for the hybrid kinetic equation with special attention to the case of possibly vanishing nonlinear diffusion in the Fokker-Planck operator. From the modelling point of view, the derivation of macroscopic traffic equations in a suitable hydrodynamic limit from the hybrid stochastic kinetic description is another completely open issue.

\section*{Acknowledgments}
The research that led to the present paper was partially supported by the MIUR-DAAD Joint Mobility Programme and by DFG HE5386/13-15.

A. T. is member of GNFM (Gruppo Nazionale per la Fisica Matematica) of INdAM (Istituto Nazionale di Alta Matematica), Italy.

G. V. and M. Z. are members of GNCS (Gruppo Nazionale per il Calcolo Scientifico) of INdAM, Italy.

M. Z. acknowledges support from ``Compagnia di San Paolo'', Turin, Italy.

\bibliographystyle{plain}
\bibliography{references}
\end{document}